\documentclass[A4paper]{article}

\usepackage{lipsum}        

\usepackage[a4paper,top=3cm,bottom=2cm,left=3cm,right=3cm,marginparwidth=1.75cm]{geometry}
\usepackage[utf8]{inputenc}
\usepackage{authblk}
\usepackage{amsmath}
\usepackage{mathtools}
\usepackage{cases}
\usepackage{bm}
\usepackage{amsfonts}
\usepackage{subfiles}
\usepackage{multicol}
\usepackage{tabularx,enumerate}
\usepackage{subcaption}
\usepackage{tikz}
\usepackage{caption}
\usepackage[absolute,overlay]{textpos}
\usepackage{ifthen}
\usepackage{tcolorbox}
\usepackage{mathrsfs}
\usepackage{scalerel}
\usepackage{graphics}

\usepackage{todonotes}

\tcbuselibrary{theorems}
\tcbuselibrary{skins}

\usepackage{algorithm}
\usepackage{algpseudocode}

\usepackage{mycommands}
\usepackage{xcolor}

\newcommand{\Tau}{\mathcal T}

\newcommand\righttwoarrow{%
        \mathrel{\vcenter{\mathsurround0pt
                \ialign{##\crcr
                        \noalign{\nointerlineskip}$\rightarrow$\crcr
                        \noalign{\nointerlineskip}$\rightarrow$\crcr
                }%
        }}%
}

\usepackage{amsthm,amssymb}
\newtheorem{theorem}{Theorem}[section]
\newtheorem{corollary}[theorem]{Corollary}
\newtheorem{lemma}[theorem]{Lemma}
\newtheorem{prop}[theorem]{Proposition}
\newtheorem{remark}[theorem]{Remark}
\newtheorem{ex}[theorem]{Example}
\newtheorem{assum}{Assumption}

\theoremstyle{definition}

\usepackage[toc,page]{appendix} 
\usepackage{hyperref}
\hypersetup{
     colorlinks   = true,
     linkcolor    = blue,
     citecolor    = red
}

\title{Inertial Quasi-Newton Methods for Monotone Inclusions: Efficient Resolvent Calculus and Primal-Dual Algorithms}
\author[*]{Shida Wang}
\author[**]{Jalal Fadili}
\author[*]{Peter Ochs}
\affil[*]{Department of Mathematics, Saarland University, Germany}
\affil[**]{Normandie Universit\'e, ENSICAEN, UNICAEN, CNRS, GREYC, France.}

\date{\today}

\begin{document}

\maketitle
\begin{abstract}
    We introduce an inertial quasi-Newton Forward-Backward Splitting Algorithm to solve a class of monotone inclusion problems. While the inertial step is computationally cheap, in general, the bottleneck is the evaluation of the resolvent operator. A change of the metric makes its computation even harder, and this is even true for a simple operator whose resolvent is known for the standard metric. In order to fully exploit the advantage of adapting the metric, we develop a new efficient resolvent calculus for a low-rank perturbed standard metric, which accounts exactly for quasi-Newton metrics. Moreover, we prove the convergence of our algorithms, including linear convergence rates in case one of the two considered operators is strongly monotone. As a by-product of our general monotone inclusion framework, we instantiate a novel inertial quasi-Newton Primal-Dual Hybrid Gradient Method (PDHG) for solving saddle point problems. The favourable performance of our inertial quasi-Newton PDHG method is demonstrated on several numerical experiments in image processing.
\end{abstract}
\section{Introduction}
Nowadays, convex optimization appears in many modern disciplines, especially when dealing with datasets of large-scale. There is a strong need for efficient optimization schemes to solve such large-scale problems. Unfortunately, the high dimensionality of the problems at hand makes the use of second-order methods intractable. A promising alternative are quasi-Newton type methods, which aim for exploiting cheap and accurate first-order approximations of the second-order information. In particular, the so-called limited memory quasi-Newton method has proved very successful for solving unconstrained large-scale problems. However, many practical problems in machine learning, image processing or statistics naturally have constraints or are non-smooth by construction.  

A problem structure that can cover a broad class of non-smooth problems in these applications is the following inclusion problem on a real Hilbert space $\mathcal{H}$:
\begin{equation}\label{vi}
    \mathrm{find}\,\,x\in \mathcal{H}\ \ \mathrm{such\ that}\ \ (A+ B)x\ni 0\,,
\end{equation}
where $A\colon\mathcal{H}\righttwoarrow{\mathcal{H}}$ is a maximally monotone operator and $\map{B}{\mathcal{H}}{\mathcal{H}}$ is a single-valued $\beta$-co-coercive operator \textcolor{black}{with $\beta>0$}. As a special case, (\ref{vi}) comprises the setting of minimization problems of the form 
\begin{equation}
    \min_{x\in\mathcal{H}} f(x)+g(x)
\end{equation} with a proper lower semi-continuous convex function $f$ and convex function $g$ with Lipschitz continuous gradient by setting $A=\partial f$ and $B=\nabla g$. 

A fundamental algorithmic scheme to tackle the problem class (\ref{vi}) is Forward-Backward Splitting (FBS). However, this algorithm may exhibit slow convergence for ill-conditioned problems, for which one would like to exploit the second-order information to adapt to the local geometry of the problem. As a computationally affordable approximation in this paper, we propose a quasi-Newton variant that takes advantage from a variable metric that is computed using first-order information only. We manage to remedy the main computational bottlenecks for this type of approaches by developing an efficient low-rank variable metric resolvent calculus. 

Our approach is inspired by the proximal quasi-Newton method in \cite{becker2019quasi}. We extend the framework proposed in \cite{becker2019quasi} to the resolvent setting with a \textcolor{black}{``$M$ $+$ rank-$r$'' or ``$M$ $-$ rank-$r$''} symmetric positive-definite variable metric. We first study the convergence of two variants of FBS algorithm with respect to this type of metrics. \textcolor{black}{One variant uses an inertial step which opens the door to acceleration. Although accelerated rates are not proved yet, numerical results show that the inertial variant yields significantly improved convergence rates.} The second variant uses a relaxation technique which \textcolor{black}{enables convergence under a weaker assumption on the variable metrics.} In analogy to the proximal calculus proposed in \cite{becker2019quasi}, we develop a resolvent calculus that allows for an efficient evaluation of resolvent operators with respect to this type of metrics by splitting the evaluation into two computationally simple steps: calculating a resolvent operator with respect to a simple metric $M$ and solving a low dimensional root-finding problem. \textcolor{black}{This allows for the incorporation of popular quasi-Newton strategies, such as the limited memory SR1 or BFGS method, in our framework.}

In order to exploit the power and to illustrate the variety of problems that can be solved via the framework in (\ref{vi}), the developed algorithms are instantiated for the following saddle point problem:  
\begin{equation}\label{SaddlePointProblem}
    \min_{x\in\mathcal{H}_1}\max_{y \in \mathcal{H}_2}\scal{Kx}{y}+g(x)+G(x) - f(y) -F(y)\,,
\end{equation}
where $g$ and $f$ are lower semi-continuous convex functions, $G$ and $F$ are convex, differentiable with Lipschitz-continuous gradients and $K$ is a bounded linear operator \textcolor{black}{between Hilbert spaces $\mathcal H_1$ and $\mathcal{H}_2$}.  
We develop a quasi-Newton Primal-Dual methods that has many potential applications in Imaging Processing, Machine Learning or Statistics \cite{chambolle2016introduction, goldstein2015adaptive}.
The numerical performance of our algorithms is tested on several experiments and demonstrates a clear improvement when our quasi-Newton methods are used. 
\subsection{Related Works}\label{relatedwork}
\textbf{Smooth quasi-Newton}.
Quasi-Newton methods are widely studied and used for optimization with sufficiently smooth objective functions \cite{wright1999numerical}. Their motivation is to build a cheap approximation to the Newton's method. If the approximation of the second-order information (Hessian) is given by a positive-definite matrix, quasi-Newton methods can be interpreted as iteratively and locally adapting the metric of the space to the objective function. The success of these methods requires a good approximation of the second-order information of the objective by using the first-order information, which is still an active research area. Recently, \cite{Rodomanov} proposed to greedily select some basis vectors instead of using the difference of successive iterates for updating the Hessian approximations. Relying on \cite{Rodomanov}, \cite{liu2021quasi} constructed an approximation of the indefinite Hessian of a twice continuously differentiable function. All methods mentioned above require sufficient smoothness of the objective functions. \\
\textbf{Non-smooth quasi-Newton.}
A broad class of optimization problems is interpreted as a composition of a smooth function $f$ and a non-smooth function $h$. To deal with the non-smoothness of $h$ efficiently, many authors consider a combination of FBS with the quasi-Newton methods. 
By using the forward-backward envelope, \cite{patrinos2014forward,stella2017forward} reinterpreted the FBS algorithm as a variable metric gradient method for a smooth optimization problem in order to apply the classical Newton or quasi-Newton method. 
For a non-smooth function $g$ as simple as an indicator function of a non-empty convex set, \cite{schmidt,pmlr-v5-schmidt09a} proposed an elegant method named projected quasi-Newton algorithm (PQN) which, however, requires either solving a subproblem or using a diagonal metric. \cite{lee2014proximal} extended PQN to a more general setting as long as the proximal operator of $h$ is simple to compute. For a class of low-rank perturbed metrics, \cite{becker2012quasi, becker2019quasi} developed a proximal quasi-Newton method with a root-finding problem as the subproblem which can be solved easily and efficiently. This method can be extended to the nonconvex setting \cite{kanzow2021globalized}. Based on \cite{becker2012quasi, becker2019quasi}, \cite{kanzow2022efficient} incorporated a limited-memory quasi-Newton update. The authors of \cite{Vavasis} developed a different algorithm to evaluate the proximal operator of the separable $l_1$ norm with respect to a low-rank metric $V=M-UU^\top$. Recently, in \cite{khanh2023generalized}, the authors proposed a generalized damped Newton type which is based on second-order generalized Hessians.

In this paper, we extend the quasi-Newton approach of \cite{becker2012quasi, becker2019quasi} from the nonsmooth convex minimization setting to monotone inclusion problems of type (\ref{vi}). \textcolor{black}{Our framework opens the door to new problems (e.g. saddle-point problems) and algorithms (e.g. primal-dual algorithms) that are beyond the reach of the approach initiated in \cite{becker2012quasi, becker2019quasi}.} It is generic to use a variable metric for solving a monotone inclusion problem (see \cite{ combettes2012variable,combettes2014variable}). Its convergence relies on quasi-Fej\'er monotonicity \cite{combettes2012variable}. \textcolor{black}{However, the efficient calculation of the resolvent remains an open problem.} Our approaches use a variable metric to obtain a quasi-Newton method with an efficient resolvent calculus. \textcolor{black}{Regularized Newton-type methods in continuous time, both for convex optimization and monotone inclusions have been studied in a series of papers by Attouch and his co-authors: \cite{abbas2014newton,attouch2013global,attouch2011continuous,attouch2015dynamic}.} \textcolor{black}{Time discretization of these dynamics gives algorithms providing insight into regularized Newton's method for solving monotone inclusions (see \cite{abbas2014newton, attouch2011continuous,attouch2015dynamic}). In \cite{attouch2015dynamic}, a relative error tolerance for the solution of the proximal subproblem is also allowed. However, in all these papers, the Hessian ends up being discretized.}\\ 
\textbf{PDHG.}
Primal-Dual Hybrid Gradient (PDHG) is widely used for solving saddle point problems of the form (\ref{SaddlePointProblem}). PDHG can be interpreted as a proximal point algorithm \cite{He} \textcolor{black}{with a fixed metric} applied \textcolor{black}{to} a monotone inclusion problem. Based on this idea, \cite{lorenz2015inertial} proposed an inertial FBS method applied to the sum of set-valued operators from which a generalization of PDHG method is derived. Also based on similar ideas, \cite{Diagonal} considers diagonal preconditioning to accelerate PDHG. Their method can be regarded as using a fixed blocked matrix as a metric. Later, \cite{liu2021acceleration} considers non-diagonal preconditioning and pointed out if a special preconditioner is chosen, that kind of preconditioned PDHG method will be a special form of the linearized ADMM. Their method requires an inner loop due to the non-diagonal preconditioning. \cite{goldstein2015adaptive} introduces an adaptive PDHG scheme which can also be understood as using a variable metric with step size tuned automatically. \textcolor{black}{However, \cite{goldstein2015adaptive} focuses on changes to the diagonal of the metric.}  Our resolvent calculus in Section~\ref{Section:calc} provides another possibility to change metrics at elements off the diagonal to deal with resolvent operators. \textcolor{black}{Additionally, in \cite{liang2016convergence}, the author investigated inexact inertial variable proximal point algorithm with a different condition on the inertial step and distinct assumptions on error terms compared to ours.}   


\section{Preliminaries}
Let us recall some essential notations and definitions. Let $\mathcal{H}$ be a Hilbert space equipped with inner product $\scal{\cdot}{\cdot}$ and induced norm $\norm{\cdot}=\sqrt{\scal{\cdot}{\cdot}}$. \textcolor{black}{The symbols $\weakto$ and $\to$ respectively denote weak and strong convergence. $\ell_+^1(\N)$ is the set of all summable sequences in $[0,+\infty)$.} An operator $K\in \mathcal{B}(\mathcal{D},\mathcal{H})$ is a linear bounded mapping from a Hilbert space $\mathcal{D}$ to $\mathcal{H}$. The adjoint of $K$ is denoted by $K^*$. We abbreviate $\mathcal{B}(\mathcal{H},\mathcal{H})$ to $\mathcal{B}(\mathcal{H})$. We define $\mathcal{S}(\mathcal{H})\coloneqq\{M\in\mathcal{B}(\mathcal{H})\vert M=M^*\}$ and the identity operator by $\opid \in \mathcal{S}(\mathcal{H})$. Without ambiguity, we also use the notation $\norm[]{M}$ for the operator norm of $M\in \mathcal{S}(\mathcal{H})$ with respect to $\norm[]{\cdot}$. The partial ordering on $\mathcal{S}(\mathcal{H})$ is given by 
\begin{equation}
    (\forall U\in \mathcal{S}(\mathcal{H}))(\forall V\in\mathcal{S}(\mathcal{H}))\colon\quad U \succeq V\iff (\forall x \in \mathcal{H})\colon\ \scal{U x}{x}\geq \scal{V x}{x}\,.
\end{equation}
For $\sigma\in[0,+\infty)$, we introduce $\mathcal{S}_\sigma(\mathcal{H})\coloneqq\{U\in\mathcal{S}(\mathcal{H})\vert U\succeq\sigma\opid)\}$. Similarly, we introduce $\mathcal{S}_{++}(\mathcal{H})\coloneqq\{U\in\mathcal{S}(\mathcal{H})\vert U\succ0\}$. \textcolor{black}{In particular, $\mathcal{S}_{++}(\R^n)$ denotes the set of $n\times n$ real symmetric positive definite matrices.} The norm $\norm[M]{\cdot}$ is defined by $\sqrt{\scal{M\cdot}{\cdot}}$ for $M\in\mathcal{S}_{++}(\mathcal{H})$. \textcolor{black}{We say $Q\in\mathcal{S}_{0}(\mathcal{H})$ has finite rank $r$ if $r=\mathrm{dim}(\mathrm{im}(Q))$. Then there are linearly independent vectors $u_i$ such that
$\map{Q}{\mathcal{H}}{\mathcal{H}},\ x\mapsto \sum_{i=1}^r \scal{x}{u_i}u_i$. As a consequence, $Q=UU^*$ where $\map{U}{\R^r}{\mathrm{im}(Q)},\,\alpha\mapsto U\alpha\coloneqq \sum^r_{i=1}\alpha_iu_i$ is an isomorphism defined by $(u_i)_{i=1,\ldots,r}$.}

A set valued operator $A\colon\mathcal{H}\righttwoarrow{\mathcal{H}}$ is defined by its graph
\begin{equation*}
    \Graph A \coloneqq \{(x,y)\in \mathcal{H}|x\in \mathrm{Dom}(A), y\in Ax\}\,,
\end{equation*}
and has a domain given by
\begin{equation*}
    \mathrm{Dom}(A)\coloneqq\{ x\in \mathcal{H}|Ax \neq \emptyset\}.
\end{equation*}
Given two set-valued operators $A,B\colon\mathcal{H}\righttwoarrow{\mathcal H}$, we define $A+B\colon\mathcal{H}\righttwoarrow{\mathcal{H}}$ as follows:
\[
\begin{split}
    \mathrm{Dom}(A+B) &= \mathrm{Dom}(A)\cap \mathrm{Dom}(B)\,,\\
    (A+B)x &= Ax + Bx\coloneqq \{y\in\mathcal{H}\vert \exists y_1\in Ax,\exists y_2\in Bx\,\  \mathrm{such\ that}\,\ y=y_1+y_2 \}\,.\\
\end{split}
\]
The inverse of $A$ is denoted by $A^{-1}$ given by $A^{-1}(y)\coloneqq\{x\in\mathcal{H}\vert y\in Ax\}$ and the zero set of $A$ is denoted by $\mathrm{zer}(A+B)\coloneqq\{ x\in\mathcal{H}\vert (A+B)x\ni 0\}$. We say that $A$ is $\gamma_A$-strongly monotone with modulus $\gamma_A\geq 0$ with respect to norm $\norm[]{\cdot}$, if $\scal{x-y}{u-v}\geq \gamma_A\norm[]{x-y}^2$ for any pair $(x,u)\,,(y,v)\in\Graph A$. We say that \textcolor{black}{a single valued} $B$ is $\beta$-co-coercive with respect to norm $\norm[]{\cdot}$, if 
$\scal{x-y}{u-v}\geq \beta\norm[]{u-v}^2$ for any pair $(x,u)\,,(y,v)\in\Graph B$. The resolvent of $A\colon\mathcal{H}\righttwoarrow{\mathcal H}$ with respect to metric $M\in \mathcal{S}_{++}(\mathcal{H})$ is defined as
\begin{equation}
    J^M_A \coloneqq (\opid + M^{-1}A)^{-1}\text{ and we set}\ J_A\coloneqq J_A^I\ \text{for the identity mapping}\,I\,,
\end{equation}
which, as shown for example in \cite{bauschke:hal-00643354}, enjoys the following properties. 
\begin{prop}
Let $A\colon\mathcal{H}\righttwoarrow{\mathcal H}$ be maximally monotone, $M\in\mathcal{S}_{++}(\mathcal{H})$ and $y\in \mathcal{H}$. Then, the following holds 
\begin{equation}
  y = J^M_{\gamma A}(x)\iff x\in y +\gamma M^{-1}Ay\iff x-y\in \gamma M^{-1}Ay \iff (y,\gamma^{-1}M(x-y))\in \Graph A\,.  
\end{equation}
\end{prop}

\begin{prop}\label{SeqClosed}
Let $A:\mathcal{H}\righttwoarrow{\mathcal H}$ be maximally monotone. Then for every sequence $(x_k,u_k)_{k\in\mathbb{N}}$ in $\Graph A$ and every $(x,u)\in\mathcal{H}\times\mathcal{H}$, if $x_k\weakto x$ and $u_k\to u$, then $(x,u)\in\Graph A$.
\end{prop}
\begin{lemma}\label{lem:JMT}
Let $T\colon\mathcal{H}\righttwoarrow{\mathcal H}$ be a maximally monotone operator and let $M\in \mathcal{S}_{++}(\mathcal{H})$. Then, for any $z\in \mathcal{H}$, we have
$J^M_T(z)= M^{-1/2}\circ J_{M^{-1/2}TM^{-1/2}}\circ M^{1/2}(z )\,.$
\end{lemma}
\begin{proof}
    See Appendix~\ref{proof:JMT}.
\end{proof}

\begin{lemma}\label{lemma:JMALip}
    If $A$ is strongly monotone with modulus $\gamma_A\geq 0$ and $M\in\mathcal{S}_{++}$, then $J_A^{M}$ is Lipschitz continuous with respect to  $\norm[M]{\cdot}$ with constant ${1}/(1+\tfrac{\gamma_A}{C})\in (0,1]$ for any $C$ satisfying $\norm[]{M}\leq C<\infty$. 
\end{lemma}
\begin{proof}
    See Appendix~\ref{PRoof:JMA}.
\end{proof}

\begin{prop}[Variable Metric quasi-Fej\'er monotone sequence \cite{combettes2012variable}]\label{fejer}
Let $\sigma\in (0,+\infty)$, let $\map{\phi}{[0,+\infty)}{[0,+\infty)}$ be strictly increasing and such that $\lim_{t\to+\infty}\phi(t)=+\infty$, let $(M_k)_{k\in\mathbb{N}}$ be in $\mathcal{S}_\sigma(\mathcal{H})$, let $C$ be a nonempty subset of $\mathcal{H}$, and let $(x_k)_{k\in\mathbb{N}}$ be a sequence in $\mathcal{H}$ such that 
\begin{equation}\label{fejerseq}
\begin{split}
    (\exists (\eta_k)_{k\in\mathbb{N}}\in \ell_+^1(\mathbb{N}))(\forall z \in C)(\exists (\epsilon_k)_{k\in\mathbb{N}}\in \ell_+^1(\mathbb{N}))&(\forall k \in \mathbb{N})\colon\\
    &\phi(\norm[M_{k+1}]{x_{k+1}-z}) \leq (1+\eta_k)\phi(\norm[M_k]{x_k-z})+\epsilon_k\,.
\end{split}
\end{equation}
\begin{enumerate}
    \item[(a)] Then $(x_k)_{k\in\mathbb{N}}$ is bounded and, for every $z\in C$, $(\norm[M_k]{z_k-z})_{k\in\mathbb{N}}$ converges.
    \item[(b)] If additionally, there exists $M\in\mathcal{S}_\sigma(\mathcal{H})$ such that $M_k\to M$ pointwisely, as is the case when
\begin{equation}
    \sup_{k\in\mathbb{N}}\norm[]{M_k}< +\infty\quad \mathrm{and}\quad (\exists(\eta_k)_{k\in\mathbb{N}}\in \ell_+^1(\mathbb{N}))(\forall k\in N)\colon\,  (1+\eta_k)M_k\succeq M_{k+1}\,,
\end{equation} then $(x_k)_{k\in \mathbb{N}}$ converges weakly to a point in $C$ if and only if every weak sequential cluster point of $(x_k)_{k\in\mathbb{N}}$ lies in $C$.
\end{enumerate}
\end{prop}




A key result for our resolvent calculus in Section~\ref{Section:calc} is the following Attouch-Th\'era abstract duality principle.
\begin{lemma}[A duality result for operators \cite{Attouch}]\label{attouch}
Let $A\colon\mathcal{H}\righttwoarrow{\mathcal H}$ be an operator such that $A^{-1}$ is single-valued and let $\map{B}{\mathcal{H}}{\mathcal{H}}$ be a single-valued operator. Then, the following holds for $x,u\in \mathcal{H}$:
\begin{equation}
    \begin{cases}
    0\in Ax + Bx\\
    0\in B^{-1}u  - A^{-1}(-u)
    \end{cases}
    \iff
    \begin{cases}
    x\in B^{-1}u\\
    -u \in Ax
    \end{cases}
    \iff
    \begin{cases}
    Bx = u\\
    x = A^{-1}(-u)
    \end{cases}\,.
\end{equation}
Moreover, if there exists $x\in \mathcal{H}$ such that $0\in Ax + Bx$ or there exists $u\in \mathcal{H}$ such that $0\in B^{-1}u- A^{-1}(-u)$, then there exists a unique primal-dual pair $(x,u)$ that satisfies the equivalent conditions above.
\end{lemma}

We also need the following lemma which was stated as \cite[Lemma 2.2.2]{Polyak}.
\begin{lemma}\label{lemmaPolyak}
    Let $C_k\geq 0$ and let
    \begin{equation}
    \begin{split}
        &C_{k+1}\leq (1+\nu_{k})C_k + \zeta_k,\quad \nu_k\geq0,\ \zeta_k\geq 0, \\
        &\sum_{k\in\mathbb{N}}\nu_k<\infty,\quad\sum_{k\in\mathbb{N}}\zeta_k<\infty.\\
    \end{split}
    \end{equation}
    Then, $C_k$ converges to a non-negative limit.
\end{lemma}

\section{Inertial Quasi-Newton FBS for Monotone Inclusions}
\subsection{Problem and Algorithms}
We consider the monotone inclusion problem \eqref{vi}.
\textcolor{black}{In order to benefit from additional properties of $A$ and $B$, whenever available, we assume that $A$ and $B$ are strongly monotone with modulus $\gamma_A\geq 0$ and $\gamma_B\geq 0$, respectively.} Note that by setting $\gamma_A=0$ or $\gamma_B=0$, we include the general case of monotone operators that are not necessarily strongly monotone.
 In this paper, we propose two variants of an efficiently implementable quasi-Newton Forward-Backward Splitting (FBS) algorithm in Algorithm~\ref{Alg:mainAlg1} and Algorithm~\ref{Alg:mainAlg2} to solve (\ref{vi}). The main update step is a variable metric FBS step in both algorithms as shown in (\ref{hatz}) and (\ref{hatz2}), i.e., a forward step with respect to the co-coercive operator $B$, followed by a proximal point step (computation of the resolvent) of the maximally monotone operator $A$, both evaluated with an iteration dependent metric $M_k$ that is inspired by quasi-Newton methods and therefore adapts to \textcolor{black}{the local geometry of the problem}. In contrast to related works, as discussed in Section~\ref{relatedwork}, we emphasize the importance of efficiently implementable resolvent operators (see Section~\ref{Section:calc}). The two variants allow for more or less flexibility for the choice of the metric (see Section~\ref{section:0sr1}). Both variants account for potential errors in the evaluation of the forward-backward step.
 Algorithm~\ref{Alg:mainAlg1} combines a FBS step with an additional inertial step which has the potential of accelerating the convergence, as we illuminate in our numerical experiments in Section~\ref{Experiments}. It is a generalization of the algorithms in \cite{chambolle2011first,lorenz2015inertial} to a quasi-Newton variant. In \cite{lorenz2015inertial}, Lorenz and Pock proposed an inertial Forward-Backward Splitting algorithm with a fixed metric which is different from our Algorithm~\ref{Alg:mainAlg1}.  Algorithm~\ref{Alg:mainAlg2}
 combines FBS with a relaxation step in (iii) which yields convergence under weak assumptions on the metric. It generalizes the correction step introduced in \cite{He} which can be retrieved by setting $M_k\equiv M$ and $B\equiv 0$.
\begin{algorithm}[h]
\caption{Inertial Quasi-Newton Forward-Backward Splitting}\label{Alg:mainAlg1}
\begin{algorithmic}

\Require $N \geq 0$, $(\norm[]{\epsilon_k})_{k\in\mathbb{N}}$
\Ensure $z_0\in\mathcal{H}$
\State \textbf{Update for $k = 0,\ldots,N$}:
\begin{enumerate}
    \item[(i)]  Compute $M_k$ according\ to a quasi-Newton framework \textcolor{black}{(see Section~\ref{section:0sr1})}.
    \item[(ii)] Compute the inertial step (extrapolation step):
\begin{equation}\label{inertialstep}
    \bar z_{k} \gets {z}_{k} + \alpha_k({z}_k-z_{k-1})\,,
\end{equation}
\item[(iii)] and the forward-backward step:
\begin{equation}\label{hatz}
    z_{k+1} \gets J_{A}^{M_k}(\bar z_k-M_k^{-1}B\bar z_k) + \epsilon_k\,.
\end{equation}
\end{enumerate}
\State \textbf{End}
\end{algorithmic}
\end{algorithm}

\begin{algorithm}[h]
\caption{Quasi-Newton Forward-Backward Splitting with Relaxation}\label{Alg:mainAlg2}
\begin{algorithmic}

\Require $N \geq 0$, $(\norm[]{\epsilon_k})_{k\in\mathbb{N}}$
\Ensure $z_0\in\mathcal{H}$
\State \textbf{Update for $k= 0,\ldots,N$}
\begin{enumerate}
    \item[(i)] Compute $M_k$ according to a quasi-Newton framework \textcolor{black}{(see Section~\ref{section:0sr1})}.
    \item[(ii)] Compute the forward-backward step:
    \begin{equation}\label{hatz2}
    \tilde z_{k} \gets J_{A}^{M_k}(z_k-M_k^{-1}Bz_k)+\epsilon_k\,,
\end{equation}
\item[(iii)] and relaxation step:
\begin{equation}\label{eq:crelax}
    t_k = \frac{\scal{z_k-\tilde z_k}{(M_k-B)(z_k-\tilde z_k)}}{2\norm[]{(M_k-B)(z_k-\tilde z_k)}^2}\,,
\end{equation}
\begin{equation}
    z_{k+1} \gets z_k-t_k [(M_k-B)(z_k-\tilde z_k)]\,.
\end{equation}
\end{enumerate}

\State \textbf{End}
\end{algorithmic}
\end{algorithm}

\subsection{Convergence Guarantees}
In this subsection, we prove the convergence of Algorithm~\ref{Alg:mainAlg1} and Algorithm~\ref{Alg:mainAlg2}. The implementation details for the specific quasi-Newton features are deferred to Section~\ref{Section:calc}. 
\subsubsection{Algorithm~\ref{Alg:mainAlg1}: Inertial Quasi-Newton Forward-Backward Splitting}
The following convergence result is a generalization of \cite[ Theorem 3.1]{combettes2014forwardbackward} to an inertial version of variable metric Forward-Backward Splitting.
\begin{assum}\label{assumption1} Let $\sigma\in (0,+\infty)$.
$(M_k)_{k\in\mathbb{N}}$ is a sequence in $\mathcal{S}_\sigma (\mathcal{H})$ such that
\begin{equation}
\begin{cases}
    C \coloneqq\mathrm{sup}_{k\in\mathbb{N}}\norm{M_k}<\infty\,,\\
    (\exists(\eta_k)_{k\in\mathbb{N}}\in \ell_+^1(\mathbb{N}))(\forall k\in\mathbb{N})\colon\quad(1+\eta_k)M_{k}\succeq M_{k+1}\,,\\
    \end{cases}
\end{equation}
and $M_k-\tfrac{1}{2\beta}\opid\in \mathcal{S}_{\kappa}(\mathcal{H})$ for all $k\in\mathbb{N}$ and some $\kappa>0$.
\end{assum}

\begin{theorem}\label{thm:1}
Consider Problem (\ref{vi}) and let the sequence $(z_k)_{k\in \mathbb{N}}$ be generated by Algorithm~\ref{Alg:mainAlg1} where Assumption~\ref{assumption1} holds. The sequence $(\alpha_k)_{k\in \mathbb{N}}$ is selected such that $\alpha_k\in(0,\Lambda]$ with $\Lambda<\infty$ and $$\sum_{k\in\mathbb{N}}\alpha_k\max\{\norm[M_k]{z_k-z_{k-1}}, \norm[M_k]{z_k-z_{k-1}}^2\}<+\infty\,,$$ \textcolor{black}{while $(\epsilon_k)_{k\in\N}$ is a sequence in $\mathcal{H}$ such that $\sum_{k\in\N}\norm[]{\epsilon_k}<+\infty$.} Suppose that $\mathrm{zer}(A +B)\neq \emptyset$. Then $(z_k)_{k\in\mathbb{N}}$ \textcolor{black}{is} bounded and weakly \textcolor{black}{converges} to a point $z^*\in \mathrm{zer}(A+B)$, i.e. $z_k\weakto z^*$ as $k\to\infty$. 

Furthermore, if additionally we assume $\epsilon_k\equiv 0$ for any $k\in\mathbb{N}$, $\gamma_A>0$ or $\gamma_B>0$ and $M_k-\tfrac{1}{\beta}\opid\in\mathcal{S}_\kappa(\mathcal{H})$ for some $\kappa>0$, then there exist some $\xi\in(0,1)$, some $\Theta>0$ and some $K_0\in \mathbb{N}$ such that for any $k> K_0$,
\begin{equation}\label{Alg1:rate}
    \norm[M_{k}]{z_{k}-z^*}^2\leq (1-\xi)^{k-K_0}\norm[M_{K_0}]{z_{K_0}-z^*}^2+\sum_{i=K_0}^{k-1}\Theta(1-\xi)^{k-i}\alpha_i\max\{\norm[M_i]{z_i-z_{i-1}},\norm[M_i]{z_i-z_{i-1}}^2\}\,.
\end{equation}

\end{theorem}
\begin{proof}
See Appendix~\ref{App:proofthm1}.
\end{proof}

\begin{remark}
\begin{itemize}
\item \textcolor{black}{Using Lemma 3.1 (iv) from \cite{combettes2001quasi}, we deduce that the second term on the right hand of inequality (\ref{Alg1:rate}) converges to 0.}
    \item The linear convergence factor $1-\xi$ is chosen such that there exists $K_0\in\mathbb{N}$ with\\ $$(1+\eta_k)\left(\frac{1-\frac{\gamma_B}{C}}{1+\frac{\gamma_A}{C}}\right)\leq 1-\xi\quad\text{for\ all}\ k\geq K_0\,.$$
    \item The convergence rate for the strongly monotone setting can be influenced by the decay rate of $\alpha_k$. 
\begin{enumerate}
    \item[(i)] If $\alpha_k = O(q^k)$ for  $q=1-\xi$, we have convergence rate of $O(kq^k)$ for $k>K_0$ where $K_0$ is sufficiently large.
    \item[(ii)] If $\alpha_k = O(\tfrac{1}{k^2})$, we have convergence rate of $O(\tfrac{1}{k})$ for $k>K_0$ where $K_0$ is sufficiently large.
\end{enumerate}
\end{itemize}
\end{remark}

\textcolor{black}{
\begin{remark}
In practice, Assumption~\ref{assumption1} is hard to verify and restrictive, however, it is a common assumption for variable metric methods \cite{combettes2014variable}. It can be avoided by relaxation methods which we propose in the next section for this reason.
\end{remark}
}
\subsubsection{Algorithm~\ref{Alg:mainAlg2}: Quasi-Newton Forward-Backward Splitting with Relaxation}
\textcolor{black}{At the expense of a relaxation instead of an inertial step, we achieve a substantial enhancement in the flexibility of selecting the metric. This method is inspired by \cite{He}. It is worth noting that without loss of generality, we assume $\tilde z_k\neq z_k$ for all $k\in\mathbb{N}$, since otherwise $\tilde z_k$ already solves the inclusion problem after a finite number of iterations.}
\begin{assum}\label{Assumption:2}
Let $\sigma\in(0,+\infty)$. $(M_k)_{k\in\mathbb{N}}$ is a sequence in $\mathcal{S}_{\sigma}(\mathcal{H})$ and $(\epsilon_k)_{k\in\mathbb{N}}$ is a sequence in $\mathcal{H}$ such that:
\begin{enumerate}
\item[(i)] For all $k\in\mathbb{N}$, we have $(M_{k}-\tfrac{1}{\beta}\opid)\in \mathcal{S}_{c}(\mathcal{H})$, for some $c>0$,
\item[(ii)] $C\coloneqq\mathrm{sup}_{k\in\mathbb{N}}\norm{M_k}<\infty$.
\end{enumerate}
\end{assum}
\begin{theorem}\label{thm:relax}
Consider Problem (\ref{vi}), and let the sequence $(z_{k})_{k\in\mathbb{N}}$ be generated by Algorithm~\ref{Alg:mainAlg2} where Assumption~\ref{Assumption:2} holds. \textcolor{black}{The sequence $(\epsilon_k)_{k\in\N}$ in $\mathcal{H}$ satisfies $\sum_{k\in\N}\norm[]{\epsilon_k}<+\infty$.} Then $(\norm[]{z_k-z^*})_{k\in\mathbb{N}}$ is bounded for any $z^*\in \mathrm{zer}(A+B)$ and $(z_{k})_{k\in\mathbb{N}}$ weakly converges to some $z^*\in \mathrm{zer}(A+B)$, i.e. $z_k\weakto z^*$ as $k\to \infty$.

Moreover, if $\epsilon_k\equiv 0$, then $\norm[]{z_k-z^*}$ decreases for any $z^*\in \mathrm{zer}(A+B)$ as $k\to\infty$. Furthermore, if $\gamma_A>0$ or $\gamma_B>0$, then $z_k$ converges linearly:
there exist some $\xi\in(0,1)$ such that
\begin{equation}
    \norm[]{z_k-z^*}^2\leq (1- \xi)^k \norm[]{z_0-z^*}^2\,.
\end{equation}

\end{theorem}
\begin{proof}
    See Appendix~\ref{proof:thm:relax}.
\end{proof}
\begin{remark}
    The linear convergence factor is given by $$\xi=\tfrac{\delta}{2}\min\{2(\gamma_A+\gamma_B), c\}\,,\quad \text{where}\ \delta=\tfrac{c}{2(C+1/\beta)^2}\,.$$
\end{remark}
\textcolor{black}{
\begin{remark}
   In \cite{He}, authors studied a relaxed proximal point algorithm for primal-dual splitting with a fixed metric. Our setting and Algorithm~\ref{Alg:mainAlg2} here are much broader.
\end{remark}
}
\begin{remark}
It is \textcolor{black}{noteworthy} that Algorithm~\ref{Alg:mainAlg2} can be interpreted as a closed loop system which uses the previous iterates (states) to update the relaxation parameter $t_k$ and the quasi-Newton metric $M_k$, i.e., the update does not explicitly depend on $k$.
\end{remark}
\section{Resolvent Calculus for Low-Rank Perturbed Metric}\label{Section:calc}
In this section, we extend the proximal calculus of \cite{becker2019quasi} to the setting of resolvent operators $J^V_T$ with a symmetric positive definite metric \textcolor{black}{$V = M+\mathrm{s} Q$}, \textcolor{black}{where $\mathrm{s}\in\{-1,+1\}$, $M$ is symmetric positive definite and $Q$ is symmetric positive semi-definite.} \textcolor{black}{This is a key result, as it enables an efficient application of quasi-Newton methods for solving monotone inclusion problems.} \textcolor{black}{Computing the resolvent operator $J^{V}_T(z)$ involves evaluating $J^{M}_T$ at a shifted point $z-M^{-1}v^*\in\mathcal{H}$, where $v^*\in \mathcal{H}$ is derived from a root-finding problem, solvable by a semi-smooth Newton method (Algorithm~\ref{Alg:Semismooth}) or bisection method (Algorithm~\ref{Alg:bisection}).} In conclusion, if $J^M_T$ can be computed efficiently, the same is true for $J_T^V$. The result crucially relies on the abstract duality principle of Attouch-Th\'era \cite{Attouch} (see Lemma~\ref{attouch}). \textcolor{black}{We first state the abstract result in a Hilbert space $\mathcal{H}$ in Theorem~\ref{thm:resolvent} and illustrate it in Corollary~\ref{coro:resolvent} with $\mathcal H=\R^n,n\in\N$.}
\begin{theorem}\label{thm:resolvent}
Let $T\colon\mathcal{H}\righttwoarrow{\mathcal H}$ be a maximally monotone operator, \textcolor{black}{$V\coloneqq M+\mathrm{s} Q\textcolor{black}{\in\mathcal{S}_{++}(\mathcal{H})}$ where $\mathrm{s}\in\{-1,1\}$}, $M\in\mathcal{S}_{++}(\mathcal{H})$ and $Q\in\mathcal{S}_0(\mathcal{H})$ having finite rank r. Then, the resolvent operator $J^V_T$ can be computed as follows:
\begin{equation}\label{eq:mainthm}
    x^* = J^V_T(z)\quad\iff\quad\begin{cases} x^*=J^{M}_T(z\textcolor{black}{-\mathrm{s}}M^{-1}\textcolor{black}{U\alpha^*})\, \text{ and }\\\alpha^*\in\R^r\, \text{ solves }\, \textcolor{black}{\ell(\alpha) = 0} \\
     \text{ where } \ell(\alpha)\coloneqq U^* \textcolor{black}{Q^{-1}} U\alpha + U^*(z - J^M_T(z\textcolor{black}{-\mathrm{s}} M^{-1}U\alpha))=0\,,\\\end{cases}
\end{equation}
where $\map{U}{\R^r}{\mathrm{im}(Q)},\,\alpha\mapsto U\alpha\coloneqq \sum^r_{i=1}\alpha_iu_i$ is an isomorphism defined by \textcolor{black}{any} $r$ linearly independent $u_1,...,u_r\in \mathrm{im}(Q)$. The function $\map{\ell}{\R^r}{\R^r}$ is Lipschitz continuous with constant $\norm{U^*\textcolor{black}{Q^{-1}}U}+\norm{M^{-1/2}U}^2$ and strictly monotone.
\end{theorem}
\begin{proof}
    See Appendix~\ref{PRoof:resolvent}.
\end{proof}
\textcolor{black}{
\begin{remark}\label{remark:single}
     A priori $Q^{-1}$ is set-valued, however it is easy to check that $U^*Q^{-1}U$ is single-valued (see Appendix~\ref{Proof:remark-single valued}). We define
     $\map{Q^+}{\mathrm{im}(Q)}{\mathrm{im}(Q)}$ as the inverse of $Q$ restricted to $\mathrm{im}(Q)$ which is a single-valued mapping. It allows us to replace $Q^{-1}$ by $Q^+$ in \eqref{eq:mainthm}.
\end{remark}
}
\textcolor{black}{In finite dimensions, Theorem~\ref{thm:resolvent} simplifies to the following corollary.}
\textcolor{black}{
\begin{corollary}\label{coro:resolvent}
Let $T\colon\R^n\righttwoarrow{  \R^n }$ be a maximally monotone operator and consider $V\coloneqq M+\mathrm{s} Q\textcolor{black}{\in\mathcal{S}_{++}(\R^n)}$ where $\mathrm{s}\in\{-1,+1\}$, $M\in\mathcal{S}_{++}(\R^n)$ and $Q\in\mathcal{S}_0(\R^n)$. Let $Q =UU^\top$ where $\map{U}{\R^r}{\R^n}$ is a matrix of full rank $r$ with $r\leq n$. Then, the resolvent operator $J^V_T$ can be computed as follows:
\begin{equation}\label{eq:la}
    x^* = J^V_T(z)\quad\iff\quad\begin{cases} x^*=J^{M}_T(z \textcolor{black}{-\mathrm{s}} M^{-1}U\alpha^*)\, \text{ and }\\\alpha^*\in\R^r\, \text{ solves }\, \textcolor{black}{\ell(\alpha) = 0}\\ \textcolor{black}{\text{ where }\,\ell(\alpha)\coloneqq \alpha + U^\top(z - J^{M}_T(z-\mathrm{s} M^{-1}U\alpha))\,.}\end{cases}
\end{equation}
The solution $\alpha^*$ is the unique root of $\map{\ell}{\R^r}{\R^r}$ which is Lipschitz continuous with constant $1+\norm{M^{-1/2}U}^2$ and strictly monotone.
\end{corollary}
}
\begin{proof}
See Appendix~\ref{Proof:coro:resolvent}.
\end{proof}
\textcolor{black}{If $r\ll n$, then we have a so-called low rank perturbed metric $M+\mathrm{s} UU^\top$, which leads to a root-finding problem of low dimension $r$. Together with the simple metric $M$ with respect to which the resolvent operator is easy to evaluate, it leads to an efficient evaluation of a resolvent operator with respect to a low rank perturbed metric.
}

\subsection{Solving the Root-Finding Problem}
The efficiency of the reduction in Theorem~\ref{thm:resolvent} relies also on the solution of a root-finding problem which we discuss thoroughly in this subsection. \textcolor{black}{We consider the space $\R^r$ and the root-finding problem with $\map{\ell}{\R^r}{\R^r}$.} \textcolor{black}{In several instances, the root-finding problem can be solved exactly, for instance, with $T=\partial g$ for special functions $g$ as enumerated in \cite[Table 3.1]{becker2019quasi}. In such cases, the root-finding problem simplifies to one involving the proximal operator rather than the resolvent. Similarly, when \textcolor{black}{$J^{M}_T$} can be represented as a composition of proximal mappings with respect to these special functions, the associated root-finding problem can be exactly solved.} 
\textcolor{black}{In situations where this subproblem cannot be exactly solved, we employ a semi-smooth Newton approach which enjoys local super-linear convergence.}  \textcolor{black}{Therefore, our Algorithms~\ref{Alg:mainAlg1} and~\ref{Alg:mainAlg2} have also modeled computational errors arising from inexactly solving the root-finding problem.}
To narrow down the neighborhood of the sought root for $r=1$, we complement the semi-smooth Newton strategy by a bisection method in Section~\ref{section:bi}. \textcolor{black}{For cases where $r\geq 1$, a globalization strategy is available as shown in \cite{solodov1999globally}}.
\subsubsection{Semi-smooth Newton Methods}
In order to solve $\ell(\alpha)=0$ in (\ref{eq:mainthm}) efficiently, we employ a semi-smooth Newton method. A locally Lipschitz function is called semi-smooth if its Clarke Jacobian defines a Newton approximation scheme \cite[Definition 7.4.2]{facchinei2003finite}. \textcolor{black}{If $\ell(\alpha)$ is semi-smooth and any element of the Clarke Jacobian $\partial ^C\ell(\alpha^*)$ is non-singular, then the inexact semi-smooth Newton method outlined in Algorithm~\ref{Alg:Semismooth} (analogous to \cite{becker2019quasi}) can be applied.}
\begin{algorithm}[htp]
\caption{Semi-smooth Newton method to solve $\ell(\alpha)=0$}\label{Alg:Semismooth}
\begin{algorithmic}

\Require A point $\alpha_0\in\R^r$. $N$ is the maximal number of iterations.
\State \textbf{Update for $k = 0,\ldots,N$}:
\If{$l(\alpha_k)=0$}
    \State \textbf{stop}
\Else{}
    \State Select $G_k\in\partial^C \ell(\alpha_k)$, compute $\alpha_{k+1}$ such that
    \[\ell(\alpha_k)+G_k(\alpha_{k+1}-\alpha_k)=e_k\,,\]
    where $e_k\in\R^r$ is an error term satisfying $\norm[]{e_k}\leq \eta_k\norm[]{G_k}$ and $\eta_k\geq 0$.
\EndIf

\State \textbf{End}
\end{algorithmic}
\end{algorithm}
Semi-smoothness may seem restrictive. However, as shown in \cite{bolte2009tame}, the broad class of tame locally Lipschitz functions is semi-smooth. We refer to \cite{bolte2009tame} for the definition of tameness. Therefore, it is sufficient to ensure $\ell(\alpha)$ is tame, which is asserted if the monotone operator $T$ in $\ell(\alpha)$ is a tame map \cite{ioffe2009invitation}. In this case, the convergence result for Algorithm~\ref{Alg:Semismooth} can be adapted from \cite{becker2019quasi}.
\begin{prop}\label{semi-smooth-Newton}Let $\ell(\alpha)$ be defined as in Theorem~\ref{thm:resolvent}, where $T$ is a set-valued tame mapping. Then $\ell(\alpha)$ is semi-smooth and all elements of $\partial^C\ell(\alpha^*)$ are non-singular where $\alpha^*$ is the unique root of $\ell(\alpha)$ from (\ref{eq:mainthm}). In turn there exists $\bar \eta$ such that if $\eta_k\leq \bar \eta$ for every $k$, there exists a neighborhood of $\alpha^*$ such that for all $\alpha_0$ in that neighborhood, the sequence generated by Algorithm~\ref{Alg:Semismooth} is well defined and converges to $\alpha^*$ linearly. If $\eta_k\to 0$, the convergence is superlinear.
\end{prop}
\begin{proof}
See Appendix~\ref{Proof:semi-smooth-Newton}.
\end{proof}

\begin{ex}
If $ f$ is a tame function and locally Lipschitz, then by \cite[Proposition 3.1]{ioffe2009invitation}, $\partial f$ is a tame map.
\end{ex}
\begin{ex}
The assumption that $T$ is a tame mapping is not restrictive. For example, in PDHG setting we have a set-valued operator $T=\begin{pmatrix}\partial g&K^*\\-K&\partial f\end{pmatrix}$ as defined by (\ref{op:T}). If $g$ and $f$ are both tame functions, then $\partial g$ and $\partial f$ are tame as well \cite{ioffe2009invitation}. As a result, $T$ is a tame mapping.
\end{ex}
\subsubsection{Bisection}\label{section:bi}
\textcolor{black}{In the case where $\mathcal{H}=\R^n$ and $r=1$, we set $U=u\in \R^{n\times 1}$} and solve the root-finding problem $\ell(\alpha)=0$ via the bisection method in Algorithm~\ref{Alg:bisection}. A similar bound on the range of $\alpha^*$ as in \cite{becker2019quasi} holds.
\begin{prop}\label{prop:bound} \textcolor{black}{Consider the setting of Corollary~\ref{coro:resolvent}}. 
For $r=1$, the root $\alpha^*$ of $\ell(\alpha) =0$ in Corollary~\ref{coro:resolvent} lies in the  set $[-\zeta,\zeta]$,  where
\begin{equation}\label{bound}
    \zeta = \norm[V^{-1}]{u}(2\norm[V]{z} + \norm[V]{J^V_T(0)})\,.
\end{equation}
Moreover, if $V\in\mathcal{S}_{c}(\R^n)$ and $\norm[]{V}\leq C$, then
\begin{equation}
    \zeta = \frac{C}{c}\norm[]{u}(2\norm[]{z} + \norm[]{J^V_T(0)})\,.
\end{equation}
\end{prop}
\begin{proof}
See Appendix~\ref{Proof:bound}.
\end{proof}
\begin{algorithm}[htp]
\caption{Bisection method to solve $\ell(\alpha)=0$ when $r=1$}\label{Alg:bisection}
\begin{algorithmic}

\Require Tolerance $\epsilon \geq 0$, number of iterates $N$
\State Compute the bound $\zeta$ from (\ref{bound}), and set $k=0$.
\State Set $\alpha_{-} = -\zeta$ and $\alpha_{+}=\zeta$. 
\State \textbf{Update for $k= 0,\ldots,N$}:
\State Set $\alpha_k =\tfrac{1}{2}(\alpha_{-}+\alpha_{+})$.
\If{$\ell(\alpha_k)>0$}
    \State $\alpha_{+}\gets \alpha_k$
\Else{}
    \State $\alpha_{-}\gets \alpha_k$
\EndIf
\If{$k>1$ and $\vert \alpha_k -\alpha_{k-1}\vert < \epsilon$}
    \State return $\alpha_k$
\EndIf
\State \textbf{End}
\end{algorithmic}
\end{algorithm}
Furthermore, we can combine the semi-smooth Newton method with the bisection method. Since the semi-smooth Newton method is locally convergent, it requires a starting point in a sufficiently near neighborhood of the solution. \textcolor{black}{Using the bisection method, we can generate a sequence of points approaching the solution. When these points reach the neighborhood required for convergence of the semi-smooth Newton method, we transition to using the semi-smooth Newton method to achieve faster convergence.}
In Proposition~\ref{semi-smooth-Newton}, $\alpha_0$ is required to belong to a neighborhood of $\alpha^*$, which can be achieved by Algorithm~\ref{Alg:bisection}, i.e., we can assert to find a point $\alpha$ such that $\vert\alpha-\alpha^*\vert<\delta$ in $\log_2((2\zeta/\delta))$ steps, where $\zeta$ is as in (\ref{bound}).
\subsection{Implementation of the quasi-Newton Forward-Backward Step}
\textcolor{black}{In Algorithm~\ref{Alg:mainAlg1} and~\ref{Alg:mainAlg2}, we evaluate the resolvent mappings with respect to $M_k$ at $\bar z_k$ and $z_k$ respectively. Here, $M_k\in\mathcal{S}_{++}(\mathcal{H})$ is generated from quasi-Newton framework (see Section~\ref{section:0sr1}) and represents a low rank perturbed metric with $M_k=M_0+\mathrm{s}Q_k$ where $Q_k$ has a finite rank $r$. In the quasi-Newton framework, there exists a connection between $Q_k$ and $U_k$ in the form of $Q_k=U_kU_k^*$ for some \textcolor{black}{bounded linear isomorphism} $\map{U_k}{\R^r}{\mathrm{im}(Q_k)}$.
The computation of both formulations involves the same abstract type of calculation at some point $z$, as represented by:
\begin{equation}\label{abstract_resolvent}
    J^{V}_A(z-V^{-1}Bz)\,,
\end{equation} where $V=M+\mathrm{s}Q\in\mathcal{S}_{++}(\mathcal{H})$ with $\mathrm{s}\in\set{-1,1}$, $M\in\mathcal{S}_{++}(\mathcal{H})$ and $Q=UU^*$  for some \textcolor{black}{bounded linear isomorphism} $\map{U}{\R^r}{\mathrm{im}(Q)}$. Here, we replace $M_k$ by $V$, omit the subscript $k$ and for Algorithm~\ref{Alg:mainAlg1} and~\ref{Alg:mainAlg2} we substitute $\bar z_k$ with $z$ and $z_k$ with $z$ respectively.} \textcolor{black}{While the computation of the resolvent~\eqref{abstract_resolvent} can be significantly simplified using Theorem~\ref{thm:resolvent}, here, an efficient update formula incorporating the forward step is presented.} Instead of evaluating $J^{V}_A(z-V^{-1}Bz)$ directly, we decompose the entire task into two subproblems: a root-finding problem and an evaluation of resolvent with respect to $M$ which is assumed to be simple to compute.
\begin{prop}\label{prop:PDHG} Consider the setting of \textcolor{black}{problem}~\eqref{vi}. Let $V=M+\mathrm{s}Q\in\mathcal{S}_{++}(\mathcal{H})$  with $\mathrm{s}\in\set{-1,1}$, $M\in\mathcal{S}_{++}(\mathcal{H})$ and $Q\in\mathcal{S}_0(\mathcal{H})$ having $Q=UU^*$ for \textcolor{black}{a bounded linear isomorphism} $\map{U}{\R^r}{\mathrm{im}(Q)}$. 
Then \eqref{abstract_resolvent} can be equivalently expressed by 
\begin{equation} \label{resolventJM}
    \hat z=J^{V}_A(z-V^{-1} B z) = J_A^M(z-V^{-1} B z -\mathrm{s}  M^{-1}U\alpha^*)\,.
\end{equation}
Here, $\alpha^*\in\mathbb{R}^{r}$ is the unique root of $\map{\mathcal{L}}{\mathbb{R}^r}{\mathbb{R}^r}$,
\begin{equation}
    \mathcal{L}(\alpha^*)\coloneqq U^*(z-V^{-1} B z-J_A^M(z-V^{-1} B z -\mathrm{s} M^{-1}U\alpha^*)) +\alpha^* = 0\,.
\end{equation}
The function $\mathcal{L}$ is Lipschitz continuous with constant $1+\norm{M^{-1/2}U}^2$ and strictly monotone.
\begin{proof}
    Apply Theorem~\ref{thm:resolvent} at the shifted point $z-V^{-1}Bz$.
\end{proof}
\end{prop}
\textcolor{black}{It remains now to compute the point $z-V^{-1} B z$ which involves inverting $V$. Since $V$ has the special structure $M + \mathrm{s}Q$, we will show that the update amounts to inverting solely $M$. This is equivalent to applying the Sherman-Morrison formula.}
\begin{prop}\label{prop:PDHGnew}
Consider the setting of Proposition~\ref{prop:PDHG}. Then \eqref{abstract_resolvent} at a point $z$ can be solved by 
\begin{equation} \label{resolventJMnew}
    \hat z=J^V_A(z-V^{-1} B z)=J_A^M(z - M^{-1}Bz-\mathrm{s}  M^{-1}U\xi^*).
\end{equation}
Here, $\xi^*\in \mathbb{R}^{r}$ is the unique zero of $ \map{\mathcal{J}}{\mathbb{R}^r}{\mathbb{R}^r}$,
\begin{equation}\label{eq:root-FB}
    \mathcal J(\xi^*)\coloneqq \textcolor{black}{U^*}(z-J_A^M(z - M^{-1}Bz -\mathrm{s} M^{-1}U\xi^*)) +\xi^* = 0\,. 
\end{equation}
The function $\mathcal{J}$ is Lipschitz continuous with constant $1+\norm{M^{-1/2}U}^2$ and strictly monotone.
\end{prop}
\begin{proof}
    See Appendix~\ref{Proof:PDHGnew}.
\end{proof}
\section{A General 0SR1 Quasi-Newton Metric}\label{section:0sr1}
We note that if we set in the inclusion problem (\ref{vi}) $A\equiv 0$, and $B=\nabla f$ of some convex smooth function $f$, the update step consisting of (\ref{hatz}) and (\ref{inertialstep}) reduces to Gradient Descent when $M_k\equiv \opid$ and to the classic Newton method when $M_k = \nabla ^2 f(z_k)$. Motivated by classic Newton and quasi-Newton methods, we construct $M_k$ as an approximation of the differential of $Bz_k$ at $z_k$. We generalize the quasi-Newton method 0SR1 from differentiable $\nabla f$ (SR1 method with 0-memory) to a co-coercive operator $B$. The approximation $M_k$ shall satisfy the modified secant condition:
\begin{equation}
    M_ks_k = y_k,\quad \mathrm{where}\quad y_k=Bz_k -Bz_{k-1},\quad s_k= z_k-z_{k-1}\,. 
\end{equation}
Choose $M_0\in\mathcal{S}_{++}(\mathcal H)$ positive-definite. \textcolor{black}{If $\scal{y_k - M_0 s_k}{s_k}= 0$, we skip the update of the metric.} \textcolor{black}{If $\scal{y_k - M_0 s_k}{s_k}\neq 0$}, the update is:
\begin{equation}\label{Alg:0SR1metric}
    M_k = M_0+\mathrm{s}\textcolor{black}{ U_kU_k^*},\quad \textcolor{black}{\hat{u}_k} = (y_k-M_0s_k)/\sqrt{|\scal{y_k - M_0 s_k}{s_k}|},
\end{equation}
where \textcolor{black}{$\map{U_k}{\R}{\mathcal{H}}$, $\alpha\mapsto U_k\alpha \coloneqq\alpha \sqrt{\gamma_k} \textcolor{black}{\hat{u}_k}$ is a bounded linear mapping}, $\gamma_k\in[0,+\infty)$ needs to be selected such that $M_k$ is positive-definite and \textcolor{black}{$\mathrm{s}\in\set{-1,+1}$} depends on the sign of $\scal{y_k - M_0 s_k}{s_k}$. If $\scal{y_k - M_0 s_k}{s_k}>0$, we use
\begin{equation}
    M_k = M_0+  \textcolor{black}{U_kU_k^*},
\end{equation}
and, if $\scal{y_k - M_0 s_k}{s_k}<0$, we use
\begin{equation}
    M_k = M_0- \textcolor{black}{U_kU_k^*}.
\end{equation}

\textcolor{black}{
In the rest of this section, we provide conditions under which $M_k$ defined in \eqref{Alg:0SR1metric} meets Assumption~\ref{assumption1} and Assumption~\ref{Assumption:2}. \textcolor{black}{We introduce a bounded linear mapping $\map{u_k}{\R}{\mathcal{H}}$, $\alpha\mapsto \alpha \textcolor{black}{\hat{u}_k}$ where $\textcolor{black}{\hat{u}_k}\in\mathcal{H}$ is  defined by \eqref{Alg:0SR1metric}. For convenience, we will also equate $U_kU_k^*$ with $\gamma_k u_ku_k^*$.} We start with Assumption~\ref{Assumption:2}.
\begin{lemma}\label{lemma:PDHGcond1}
Let $M_0$ be symmetric positive-definite and $B$ is $\beta$-co-coercive. Suppose that $M_0-\frac{1}{\beta}\in\mathcal{S}_c(\mathcal{H})$ for $c > 0$. 
\begin{enumerate}[(i)]
    \item If $M_k = M_0 + \gamma_k u_ku_k^*$, where $0 < \inf_{k \in \N} \gamma_k \leq \sup_{k \in \N} \gamma_k < +\infty$, then for all $k\in\mathbb{N}$, $M_k-\frac{1}{\beta}\in\mathcal{S}_c(\mathcal{H})$ and $M_k$ verifies Assumption~\ref{Assumption:2}.
    \item If $M_k = M_0 - \gamma_k u_ku_k^*$, where $0 < \inf_{k \in \N} \gamma_k \leq \sup_{k \in \N} \gamma_k < \frac{c^2}{(1/\beta + \norm{M_0})^2}$, then $M_k$ is positive-definite and satisfies Assumption~\ref{Assumption:2}.
\end{enumerate}
\end{lemma}
\begin{proof}
See Appendix~\ref{Proof:PDHGcond1}.
\end{proof}
}
\textcolor{black}{
\begin{lemma}\label{Lemma:PDHGcond2}
Let $M_0$ be symmetric positive-definite and $B$ is $\beta$-co-coercive. Suppose that $M_0-\frac{1}{\beta}\in\mathcal{S}_c(\mathcal{H})$ for $c > 0$. Let $(\eta_k)_{k \in \N}\in \ell_+^1(\mathbb{N})$.
\begin{enumerate}[(i)]
    \item Case $M_k = M_0 + \gamma_k u_ku_k^*$: take $\gamma_k = \frac{\eta_k}{\norm{u_k}^2}(1/\beta+c)$, and assume that $(\eta_k)_{k \in \N}$ is non-increasing. Then for all $k\in\mathbb{N}$, $M_k-\frac{1}{\beta}\in\mathcal{S}_c(\mathcal{H})$ and $M_k$ verifies Assumption~\ref{assumption1}.
    \item Case $M_k = M_0 - \gamma_k u_ku_k^*$: take $\gamma_k = \frac{\kappa\eta_k}{\norm{u_k}^2}(1/\beta+c)$, where $\kappa \in ]1/(1+\beta c),1[$, and assume that $\sup_{k \in \N} \eta_k \leq 1/\kappa-1$. Then $M_k$ is positive-definite and satisfies Assumption~\ref{assumption1}.
\end{enumerate}
\end{lemma}
\begin{proof}
See Appendix~\ref{Proof:PDHGcond2}.
\end{proof}
}
\textcolor{black}{
\begin{remark}
    The proof of Lemma~\ref{Lemma:PDHGcond2} does not rely on co-coercivity of $B$ unlike that of Lemma~\ref{lemma:PDHGcond1}. If one uses that property, and more precisely the bounds in \eqref{eq:cocoercivebnd}, then the choice of $\gamma_k$ can be made independent of $u_k$ in the form $\gamma_k = C\eta_k$, where $C$ is a positive constant that depends only on $(\beta,c,\norm{M_0})$.
\end{remark}
}


\section{Inertial Quasi-Newton PDHG for Saddle-point Problems}\label{section:IQNPDHG}

In this section, we consider a min-max problem as follows:
\begin{equation}\label{minmax}
    \min_{x\in \mathcal{H}_1 }\max_{y\in \mathcal{H}_2} g(x) + G(x) + \scal{Kx}{y} - f(y) - F(y)
\end{equation}
with a linear mapping $\map{K}{\mathcal{H}_1}{\mathcal{H}_2}$ and proper lower semi-continuous convex functions $\map{g}{\mathcal{H}_1}{\overline{\R}\coloneqq \R\cup\{
+\infty\}}$ and $\map{f}{\mathcal{H}_2}{\overline{\R}}$. \textcolor{black}{Additionally, we consider continuously differentiable convex functions $\map{G}{\mathcal{H}_1}{\R}$ and $\map{F}{\mathcal{H}_2}{\R}$ with Lipschitz continuous gradients.} This problem can be expressed as a special monotone inclusion problem from which we derive in Algorithm~\ref{Alg:IQN-PDHG} an inertial quasi-Newton Primal-Dual Hybrid Gradient Method (PDHG) and in Algorithm~\ref{Alg:RQN-PDHG} a quasi-Newton PDHG with relaxation step as special applications of Algorithm~\ref{Alg:mainAlg1} and Algorithm~\ref{Alg:mainAlg2} respectively (see Proposition~\ref{Prop:ppp}). By Fermat's rule, the optimality condition for (\ref{minmax}) is the following inclusion problem:
\begin{equation}\label{op:T}
    0\in Az + Bz\quad \mathrm{with}\,z=\begin{pmatrix}x\\y\end{pmatrix},\,\mathrm{where}\,Az = \begin{pmatrix}\partial g(x)+ K^*y \\ -Kx + \partial f(y)\end{pmatrix} \mathrm{and}\, Bz = \begin{pmatrix} \nabla G(x)\\\nabla F(y)\end{pmatrix}\,.
\end{equation}

\begin{algorithm}[h]
\caption{Inertial quasi-Newton PDHG}\label{Alg:IQN-PDHG}
\begin{algorithmic}

\Require $N \geq 0$, $(\norm[]{\epsilon_k})_{k\in\mathbb{N}}\in\ell_+^1(\mathbb{N})$ 
\State \textbf{Update for $k=0,\ldots,N$}:
\begin{enumerate}
    \item[(i)]Compute $M_k$ according to a quasi-Newton framework.
    \item[(ii)] Compute the inertial step with parameter $\alpha_k$:
\begin{equation}
    \begin{split}
        \bar x_k & = x_k + \alpha_k (x_k - x_{k-1})\,,\\
        \bar y_k & = y_k + \alpha_k (y_k - y_{k-1})\,.\\
    \end{split}    
\end{equation}
\item[(iii)] Compute the main quasi-Newton PDHG step:
\begin{equation}\label{hat-zPDHG}
    \begin{split}
        x_{k+1} &=\mathrm{prox}^{\Tau}_{g}(\bar x_k-\Tau\nabla G(\bar x_k)-\Tau K^*\bar y_k-\mathrm{s}\Tau U_{k,x}\xi_k)+\epsilon_{k,x}\,,\\
        y_{k+1} &=\mathrm{prox}^{\Sigma}_{f}(\bar y_k -\Sigma\nabla F(\bar y_k) +\Sigma K(2 x_{k+1} -\bar x_k)-\mathrm{s} \Sigma U_{k,y} \xi_k)+\epsilon_{k,y},\\
    \end{split}
\end{equation}
where $\xi_k$ solves $\mathcal J(\xi_k) = 0 $ (see \eqref{eq:root-PDHG}) and $\epsilon_k=\begin{pmatrix}
    \epsilon_{k,x}\\\epsilon_{k,y}
\end{pmatrix}$ is the error caused by computation at the $k$-th iterate.
\end{enumerate}

    

\State \textbf{End}
\end{algorithmic}
\end{algorithm}

\begin{algorithm}[h]
\caption{Quasi-Newton PDHG with relaxation}\label{Alg:RQN-PDHG}
\begin{algorithmic}

\Require $N \geq 0$, $(\norm[]{\epsilon_k})_{k\in\mathbb{N}}\in\ell_+^1(\mathbb{N})$
\State \textbf{Update for $k=0,\ldots,N$}:
\begin{enumerate}
    \item[(i)] Compute $M_k$ according to quasi-Newton framework.
    \item[(ii)]  Compute the main quasi-Newton PDHG step:
\begin{equation}\label{relax-hat-zPDHG}
    \begin{split}
        \tilde x_{k} &=\mathrm{prox}^{\Tau}_{g}(x_k-\Tau\nabla G( x_k)-\Tau K^* y_k-\mathrm{s}\Tau U_{k,x}\xi_k)+\epsilon_{k,x}\,,\\
        \tilde y_{k} &=\mathrm{prox}^{\Sigma}_{f}( y_k -\Sigma\nabla F( y_k) +\Sigma K(2 \tilde x_{k} -x_k)-\mathrm{s} \Sigma U_{k,y} \xi_k)+\epsilon_{k,y}\,,\\
    \end{split}
\end{equation}
where $\xi_k$ solves $\mathcal J(\xi_k)= 0$ and $\epsilon_k=\begin{pmatrix}
    \epsilon_{k,x}\\\epsilon_{k,y}
\end{pmatrix}$ is the error caused by computation at the $k$-th iterate.
    \item[(iii)] Relaxation step:
\begin{equation}
    v_k\coloneqq M_k\bigg(\begin{pmatrix}x_{k}\\y_{k}\end{pmatrix}-\begin{pmatrix}\tilde x_{k}\\\tilde y_{k}\end{pmatrix}\bigg)+\begin{pmatrix}\nabla G(\tilde x_{k})\\\nabla F(\tilde y_{k})\end{pmatrix}-\begin{pmatrix}\nabla G(x_{k})\\\nabla F(y_{k})\end{pmatrix}
\end{equation}
to compute the relaxation parameter $t_k$
\begin{equation}\label{relxPDHG}
    t_k = \frac{\scal{\begin{pmatrix}x_{k}\\y_{k}\end{pmatrix}-\begin{pmatrix}\tilde x_{k}\\\tilde y_{k}\end{pmatrix}}{v_k}}{2\norm[]{v_k}^2}
\end{equation}
and update $x_k$ and $y_k$ as follows
\begin{equation}
    \begin{pmatrix}x_{k+1}\\y_{k+1}\end{pmatrix} \gets \begin{pmatrix}x_{k}\\y_{k}\end{pmatrix}-t_k v_k\,.
\end{equation}
\end{enumerate}
\State \textbf{End}
\end{algorithmic}
\end{algorithm}
\begin{prop}[PDHG method as a special proximal point algorithm \cite{He}]\label{Prop:ppp}
The update step of PDHG can be regarded as a proximal point algorithm with special metric $M$:
\begin{equation}\label{update}
    M(z_{k+1} -z_k) + Az_{k+1}\ni - Bz_k\,,
\end{equation}
\textcolor{black}{where $M = \begin{pmatrix} \Tau^{-1} & -K^*\\-K & \Sigma^{-1} \end{pmatrix}$ with two fixed operators $\Tau\in\mathcal{S}_{++}(\mathcal{H}_1)$ and $\Sigma\in\mathcal{S}_{++}(\mathcal{H}_2)$ \textcolor{black}{such that $M \succ 0$. The latter is verified when $\norm{\Sigma^{1/2}K\Tau^{1/2}}< 1$}.}
\end{prop}
\begin{remark}\label{rem:metricpdhg}
\textcolor{black}{It is straightforward to verify that 
\begin{equation*}
M \succeq (1-\norm{\Sigma^{1/2}K\Tau^{1/2}})\begin{pmatrix} \Tau^{-1} & 0\\0 & \Sigma^{-1} 
\end{pmatrix}
\end{equation*}
If $\Tau =\tau \opid$ and $\Sigma =\sigma\opid$, we can retrieve a PDHG algorithm with constant stepsizes $\tau > 0$ and $\sigma > 0$. Moreover, $M \succeq (1-\tau^{1/2}\sigma^{1/2}\norm{K})\min(\tau^{1/2},\sigma^{1/2})$, and a sufficient condition for $M \succ 0$ is that $\tau\sigma\norm{K}^2 < 1$.}  
\end{remark}
Therefore, we are able to apply Algorithm~\ref{Alg:mainAlg1} and Algorithm~\ref{Alg:mainAlg2} directly to a saddle point problem by regarding $T$ as $A$ in the inclusion problem (\ref{vi}), which will give us Algorithm~\ref{Alg:IQN-PDHG} and Algorithm~\ref{Alg:RQN-PDHG}, respectively. 
We are going to use a variable metric $M_k = M+\mathrm{s} Q_k \in \mathcal{S}(\mathcal{H}_1\times \mathcal{H}_2)$ \textcolor{black}{with $\mathrm{s}\in\set{1,-1}$} instead of $M\in \mathcal{S}(\mathcal{H}_1\times \mathcal{H}_2)$ in (\ref{update}) for the $k$-th iterate. Here, we set $Q_k = U_kU_k^*$ where $\map{U_k}{\R^r}{\mathrm{im}(Q)},\ \alpha\mapsto U_k(\alpha)\coloneqq\sum_{i=1}^r\alpha_i u_{k,i}$ is an isomorphism defined by $r$ linearly independent $u_{k,1},\ldots,u_{k,r}\in\mathcal{H}_1\times \mathcal{H}_2$ for each $k$-th iterate. We obtain the update step: Find $z_{k+1}\in\mathcal{H}_1\times\mathcal{H}_2$ such that
\begin{equation}\label{VarMetricUp}
    M_k(z_{k+1} -z_k) + Az_{k+1}\ni - Bz_k\,.
\end{equation}
In practice, Proposition~\ref{prop:PDHGnew} which is derived from our main result Theorem~\ref{thm:resolvent} turns out to be more tractable \textcolor{black}{than other ways to evaluate variable metric proximal mapping, e.g. using coordinate descent to solve a subproblem which has the same dimension as the original problem \cite{scheinberg2016practical}.} To show how to calculate the update step, using Proposition~\ref{prop:PDHGnew} in quasi-Newton PDHG in Algorithm~\ref{Alg:IQN-PDHG} and in Algorithm~\ref{Alg:RQN-PDHG}, we introduce Proposition~\ref{prop:PDHGnew2}. Let $U_k\in \mathcal{B}(\R^r,\mathcal{H}_1\times\mathcal{H}_2)$. We set $U_k=\begin{pmatrix}U_{k,x}\\U_{k,y}\end{pmatrix}$ with $U_{k,x}\in\mathcal{B}(\R^r,\mathcal{H}_1)$ and $U_{k,y}\in \mathcal{B}(\R^r,\mathcal{H}_2)$. 
The validity of the update step is verified in Proposition~\ref{prop:PDHGnew2}.
\begin{prop}\label{prop:PDHGnew2}
The update step from $\bar z_k$ ($z_k$) to $z_{k+1}$ ($\tilde z_k$) in Algorithm~\ref{Alg:IQN-PDHG} (Algorithm~\ref{Alg:RQN-PDHG}), \textcolor{black}{which is the quasi-Newton PDHG update step}, reduces to compute
\begin{equation} \label{resolventJMnew2}
    \begin{cases}
    x_{k+1} &=\mathrm{prox}^{\Tau}_{g}(\bar x_k-\Tau\nabla G(\bar x_k)-\Tau K^*\bar y_k-\mathrm{s}\Tau U_{k,x}\xi_k)\\
    y_{k+1} &=\mathrm{prox}^{\Sigma}_{f}(\bar y_k -\Sigma\nabla F(\bar y_k) +\Sigma K(2 x_{k+1} -\bar x_k)-\mathrm{s} \Sigma U_{k,y}\xi_k)
    \,.\end{cases}
\end{equation}
where, $\xi_k\in \mathbb{R}^{r}$ is the unique zero of $ \map{\mathcal{J}}{\mathbb{R}^r}{\mathbb{R}^r}$:
\begin{equation}\label{eq:root-PDHG}
\begin{split}
      \mathcal{J}(\xi)&=(U_{k,x})^*[\bar x_k-\underbrace{\mathrm{prox}_g^\Tau(\bar x_k-\Tau\nabla G(\bar x_k)-\Tau K^*\bar y_k-\mathrm{s} \Tau U_{k,x}\xi)}_{x_{k+1}(\xi)}]\\&+(U_{k,y})^*[\bar y_k-\mathrm{prox}_{f}^\Sigma(\bar y_k-\Sigma \nabla F(\bar y_k)+\Sigma K(2x_{k+1}(\xi)-\bar x_k)-\mathrm{s} \Sigma U_{k,y}\xi)]+\xi\,.\\
\end{split}
\end{equation}
\end{prop}
\textcolor{black}{
\begin{proof}
 It is a direct consequence of Proposition~\ref{prop:PDHGnew} and Proposition~\ref{Prop:ppp}.
\end{proof}
}
\begin{remark}
    By switching $\bar z_k$ with $z_k$ and $z_{k+1}$ with $\tilde z_k$, we obtain the update step for Algorithm~\ref{Alg:RQN-PDHG}. 
\end{remark}
\textcolor{black}{
\begin{remark}
In special cases, the root-finding problem can be solved exactly. For instance, when $\mathcal{T}=\tau \opid$, $\Sigma =\sigma\opid$, $K=\opid$, $g\equiv 0$ and $f(y) = \norm[1]{y}$, according to \cite{becker2019quasi}, the root-finding problem with $r=1$ is piece-wisely linear and can be solved exactly.
\end{remark}
}
We would like to emphasize that using Proposition~\ref{prop:PDHGnew2}, we can avoid the computation of \textcolor{black}{$M_k^{-1}$} in the primal and dual setting, which is a computationally significant advantage. The convergence of Algorithms~\ref{Alg:IQN-PDHG} and~\ref{Alg:RQN-PDHG} is a direct consequence of Theorems~\ref{thm:1} and~\ref{thm:relax}.
\begin{prop}[Convergence of quasi-Newton PDHG method]
    Let $M_0=M$ as defined in (\ref{update}) and $M_k=M_0+\mathrm{s} U_kU_k^*$.
    \begin{enumerate}
        \item[(i)] If $(M_k)_{k\in\mathbb{N}}$ satisfies Assumption~\ref{assumption1} (see e.g. Lemma~\ref{Lemma:PDHGcond2}), then $(x_k,y_k)$ generated by Algorithm~\ref{Alg:IQN-PDHG} converges weakly to some solution of (\ref{SaddlePointProblem}). Furthermore, if $g$ and $f$ are both strongly convex (or $G$ and $F$ are both strongly convex), then we obtain the same convergence rate as in Theorem~\ref{thm:1}.
        \item[(ii)] If $(M_k)_{k\in\mathbb{N}}$ satisfies Assumption~\ref{Assumption:2} (see e.g. Lemma~\ref{lemma:PDHGcond1}), then $(x_k,y_k)$ generated by Algorithm~\ref{Alg:RQN-PDHG} converges weakly to some solution of (\ref{SaddlePointProblem}). Furthermore, if $g$ and $f$ are both strongly convex (or $G$ and $F$ are both strongly convex) , then we obtain linear convergence.
    \end{enumerate}
\end{prop}


\section{Numerical experiments}\label{Experiments}
The algorithms that we analyze in the experiments are summarized in Table~\ref{tab:my-table}.
\begin{table}[H]
\centering
\resizebox{\columnwidth}{!}{%
\begin{tabular}{llll}
\hline
\textbf{Algorithm}         & \textbf{Algorithm Name}                                                                                   & \textbf{Metric}                                                                                                          \\ \hline
 \textcolor{black}{FBS}                       & Foward-Backward Primal-Dual Hybrid Gradient Method                                                                 &  $M$ fixed as in (\ref{update})                                                                                                  \\ \hline
 \textcolor{black}{IFBS}              & Inertial Primal-Dual Hybrid Gradient Method                                                                 & $M$ fixed as in (\ref{update})                              \\ \hline
 \hyperref[Alg:IQN-PDHG]{QN-FBS}          & \begin{tabular}[c]{@{}l@{}}quasi-Newton \\ Primal-Dual Hybrid Gradient Method Gradient\end{tabular}                   & Variable metric as in (\ref{Alg:0SR1metric}) with $M_0=M$ from \eqref{update}                                                     \\ \hline
\hyperref[Alg:RQN-PDHG]{RQN-FBS}       & \begin{tabular}[c]{@{}l@{}}Primal-Dual Hybrid Gradient Method \\ with relaxation step\end{tabular} & Variable metric as in (\ref{Alg:0SR1metric}) with $M_0=M$ from \eqref{update}                                                                      \\ \hline
\hyperref[Alg:IQN-PDHG]{IQN-FBS} & \begin{tabular}[c]{@{}l@{}}Inertial quasi-Newton \\ Primal-Dual Hybrid Gradient Method\end{tabular}          &   Variable metric as in (\ref{Alg:0SR1metric}) with $M_0=M$ from \eqref{update}                                    \\ \hline

\end{tabular}%
}
\caption{Summary of algorithms used in the numerical experiments. Details are provided within each section. }
\label{tab:my-table}
\end{table}

Note that PDHG is used interchangeably as FBS in the later experiments since PDHG is a specialization of FBS.
\subsection{TV-\texorpdfstring{$l_2$}{TEXT} deconvolution}
\label{Experiments:image}
In this experiment, we solve a problem that is used for image deconvolution \cite{bredies2018mathematical}. Given a blurry and noisy image $b\in\R^{MN}$ (interpreted as a vector by stacking the $M$ columns of length $N$), we seek to find a clean image $x\in\R^{MN}$ by solving the following optimization problem:
\begin{equation}\label{PRob:imageDconvolution}
    \min_{0\leq x\leq 255} \frac{1}{2}\norm[2]{\textcolor{black}{L}x-b}^2 + \mu\norm[2,1]{Dx}\,,
\end{equation}
where $\textcolor{black}{L}\in \R^{MN\times MN}$ is a linear operator that acts as a blurring operator and $\norm[2,1]{Dx}$ implements a discrete version of the isotropic total variation norm of $x$ using simple forward differences in horizontal and vertical direction \textcolor{black}{with $\norm[]{D}\leq 2\sqrt{2}$}. The parameter $\mu>0$ stresses the influence of the regularization term $\norm[2,1]{Dx}$ versus the data fidelity term $\tfrac{1}{2}\norm[]{\textcolor{black}{L}x-b}^2$. In order to deal with the non-smoothness, we rewrite the problem as a saddle point problem: 
\begin{equation}
    \min_{x}\max_{y} \scal{Dx}{y} + \delta_{\Delta}(x)+ \frac{1}{2}\norm[2]{\textcolor{black}{L}x - b}^2 - \delta_{\{\norm[2,\infty]{\cdot}\leq\mu\}}(y)\,,
\end{equation}
where $\Delta\coloneqq\{x\in\R^{MN}\vert 0\leq x_i\leq 255,\forall i\in\set{1,\ldots,MN}\}$. We can cast this problem into the general class of problems (\ref{minmax}) by setting $K=D$, $f = \delta_{\{\norm[2,\infty]{\cdot}\leq\lambda\}}$, $G(x)=\frac{1}{2}\norm[2]{\textcolor{black}{L}x-b}^2$, $F(p)=0$ and $g=\delta_{\Delta}$. 
\textcolor{black}{Here, $G$ is $1/\beta$-smooth with $\beta = 2/3$ as we took $\norm{\textcolor{black}{L}}^2 \leq 3/2$. Let $z_k= (x_k,y_k)$ be the primal-dual iterate sequence. Choosing $\Tau =\tau\opid=0.05\opid$ and $\Sigma =\sigma\opid=0.05\opid$, it follows from Proposition~\ref{Prop:ppp} and Remark~\ref{rem:metricpdhg} that $M-\frac{1}{\beta}\opid \succeq (16-3/2)\opid > 0$ as desired}. We compute the low-rank part $Q_k = \gamma_k u_k u_k^\top$ by (\ref{Alg:0SR1metric}) with $Bz_k=\left(\begin{smallmatrix}\textcolor{black}{L}^\top \textcolor{black}{L}x_k-\textcolor{black}{L}^\top b\\0 \end{smallmatrix}\right)$ \textcolor{black}{which is of course $\beta$-co-coercive}. This leads to a metric that affects only the primal update. In each iteration, we use the semi-smooth Newton method \textcolor{black}{(Algorithm~\ref{Alg:Semismooth})} to locate the root.
\begin{figure}[h]

\centering
\includegraphics[width=7cm]{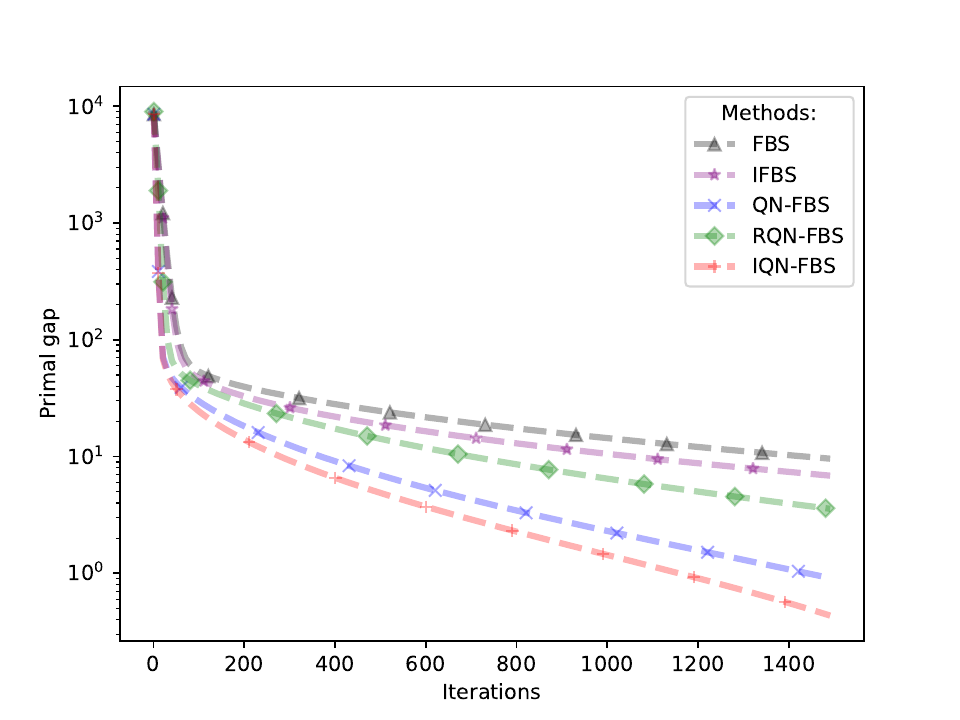}
\includegraphics[width=7cm]{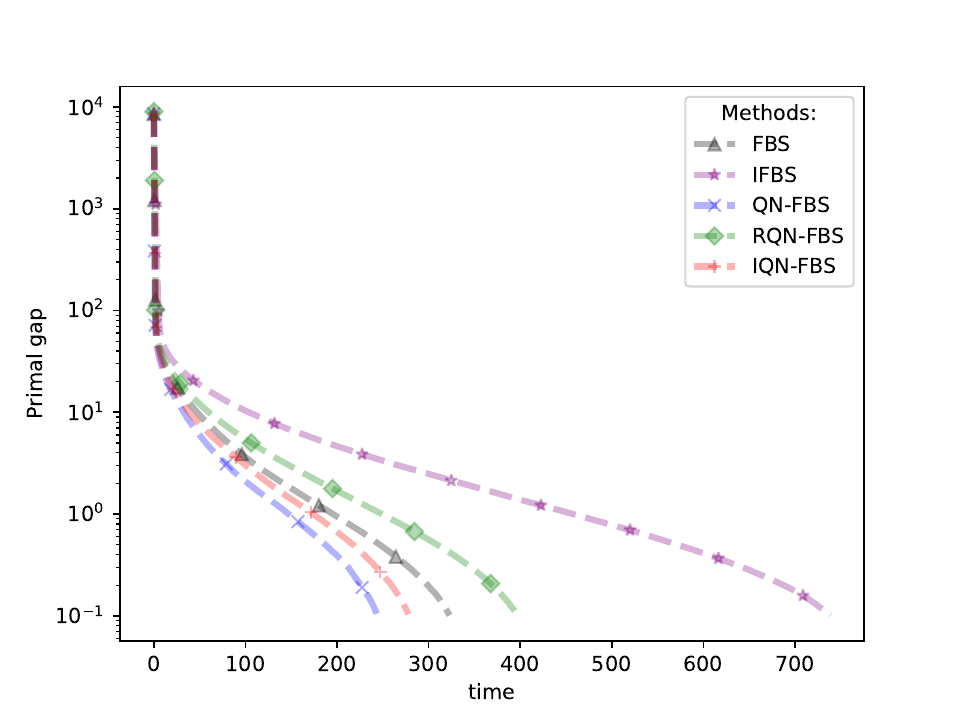}
\caption{We compare convergence of the inertial quasi-Newton PDHG (IQN-FBS) to other algorithms in the Table~\ref{tab:my-table} with $\mu = 0.0001$, $\tau =0.05$, $\sigma =0 .05$ and inertial parameter $\alpha_0 =10$, $\alpha_k = \frac{10}{k^{1.1}(\max\{\norm[]{z_k-z_{k-1}},\norm[]{z_k-z_{k-1}}^2\})}$. \textcolor{black}{The left plot depicts the convergence against the number of iterations, while the right plot shows the convergence with respect to time (seconds)}. We observe that our two quasi-Newton type algorithm QN-, and IQN-FBS clearly outperform the original FBS and IFBS algorithm.}
\label{fig:image-deconvolution}
\end{figure}

Figure~\ref{fig:image-deconvolution} shows the primal gap \textcolor{black}{against the number of iterations and against the time (seconds)}, where the optimal primal value was computed by running the original PDHG method for 10000 iterations. 
For the variable metric at iteration $k$, we fixed $\gamma_k=\min(0.8, 15/\norm[2]{u_k}^2)$. Thus, Assumption~\ref{Assumption:2} is satisfied and the convergence of RQN-FBS (Algorithm~\ref{Alg:mainAlg2}) is guaranteed. 
\textcolor{black}{Although Assumption~\ref{assumption1} can be guaranteed by appropriately defining $(\eta_{k})_{k\in\N}$ in the view of Lemma~\ref{Lemma:PDHGcond2}, we choose not to include the condition that $(1+\eta_k)M_k\succeq M_\kp$ with $(\eta_{k})_{k\in\N}\in\ell_+^1(\N)$ in the numerical experiments. This decision comes from the challenge of setting $(\eta_{k})_{k\in\N}$ in practice and the concern on imposing stringent conditions on $(\eta_{k})_{k\in\N}$ in advance, which may undermine the advantages of utilizing a variable metric.}
In this practical problem, we still observe the convergence of IQN-FBS (Algorithm~\ref{Alg:mainAlg1}). 
We notice that our quasi-Newton type algorithms IQN-FBS, RQN-FBS and QN-FBS are much faster than original FBS algorithm and inertial FBS (IFBS) \textcolor{black}{according to the left plot in Figure~\ref{fig:image-deconvolution}}. This can be explained by the fact that the Hessian of $G(x)$ is not the identity and thus by applying our quasi-Newton SR1 methods, we can adapt the metric to the local geometry of the objective. 
\textcolor{black}{Even though we are concerned about the cost of solving the root-finding problem, the right plot demonstrates the additional iterations can pay off by showing the convergence against the time. Quasi-Newton type algorithms (QN-FBS,IQN-FBS) achieve higher accuracy than FBS and IFBS after running the same time.}
  
\subsection{TV-\texorpdfstring{$l_2$}{TEXT} deconvolution with infimal convolution type regularization}
A source of optimization problems that fits (\ref{minmax}) is derived from the following:
\begin{equation}
    \min_{x\in\mathbb{R}^n} g(x) + G(x) + (f\square h)(Dx),
\end{equation}
where $f \square h(\cdot)\coloneqq \inf_{v\in\mathbb{R}^m}f(v)+h(\cdot-v)$ denotes the infimal convolution of $f$ and $h$. As a prototypical image processing problem, we define a regularization term as infimal convolution between the total variation norm and a weighted squared norm, i.e. $g=0$, $G(x)=\frac{1}{2}\norm[]{\textcolor{black}{L}x-b}^2$, $h(\cdot)= \frac{1}{2}\norm[]{W\cdot}^2$ and $f(\cdot)=\mu\norm[2,1]{\cdot}$. This yields the problem:
\begin{equation}\label{inficonv:Primal}
    \min_{x}\frac{1}{2}\norm[]{\textcolor{black}{L}x-b}^2+\mu R(Dx)
\end{equation}
where $R(\cdot)\coloneqq \inf_{v\in\mathbb{R}^m}\norm[2,1]{v}+\frac{1}{2\mu}\norm[]{W(\cdot-v)}^2$, $W$ is a diagonal matrix of weights which is given to favor discontinuities along image edges and $\textcolor{black}{L}$, $b$, $D$ are defined as in the first experiment. In practice, $W$ can be computed by additional edge finding steps or by extra information. Here, we select $W$ such that $\tfrac{1}{6}\leq \norm[]{W}^2\leq 1$. The optimization problem (\ref{inficonv:Primal}) given in primal type can be converted into the saddle point problem:
\begin{equation}\label{inficonv:saddle}
    \min_{x}\max_{y} \scal{Dx}{y} + \frac{1}{2}\norm[]{\textcolor{black}{L}x-b}^2 - \delta_{\{\norm[2,+\infty]{\cdot}\leq \mu\}}(y) - \frac{1}{2}\norm[]{W^{-1}y}^2\,.
\end{equation}
We compute the low-rank part $Q_k=\gamma_ku_ku_k^\top$ by (\ref{Alg:0SR1metric}) with $Bz_k = \left(\begin{smallmatrix} \textcolor{black}{L}^\top \textcolor{black}{L}x_k-\textcolor{black}{L}^\top b\\(W^{-1})^*W^{-1}y_k\end{smallmatrix}\right)$ which leads to a metric that affects both primal and dual update. \textcolor{black}{Here, $B$ is $\beta$-co-coercive with $\beta \geq 1/6$}. In each iteration, we combine the bisection (Algorithm~\ref{Alg:bisection}) and the semi-smooth Newton method (Algorithm~\ref{Alg:Semismooth}) to locate the root.

Figure~\ref{fig:image-deblurring-infimalconvolution} also shows the primal gap where the optimal primal value was still computed by running original PDHG for 10000 iterates. For the variable metric at iterate $k$, we fixed 
\textcolor{black}{$\gamma_k=0.64$, $\Tau=\tau\opid=0.01\opid$ and $\Sigma=\sigma\opid=0.01\opid$}. \textcolor{black}{Thus, by Lemma~\ref{lemma:PDHGcond1},  Assumption~\ref{Assumption:2} is satisfied.} 
\textcolor{black}{As for Assumption~\ref{assumption1}, we follow the same strategy as in the previous experiment and drop the condition that $(1+\eta_k)M_k\succeq M_\kp$with $(\eta_{k})_{k\in\N}\in\ell_+^1(\N)$ in this numerical experiment.}
We can observe from Figure~\ref{fig:image-deblurring-infimalconvolution}: IQN-FBS is still the fastest one. Moreover, the two quasi-Newton type methods (IQN-FBS and QN-FBS) converge faster than IFBS and FBS. Inertial methods (IQN-FBS, IFBS) are slightly faster respectively than QN-FBS and FBS.
\begin{figure}[H]

\centering
\includegraphics[width=7cm]{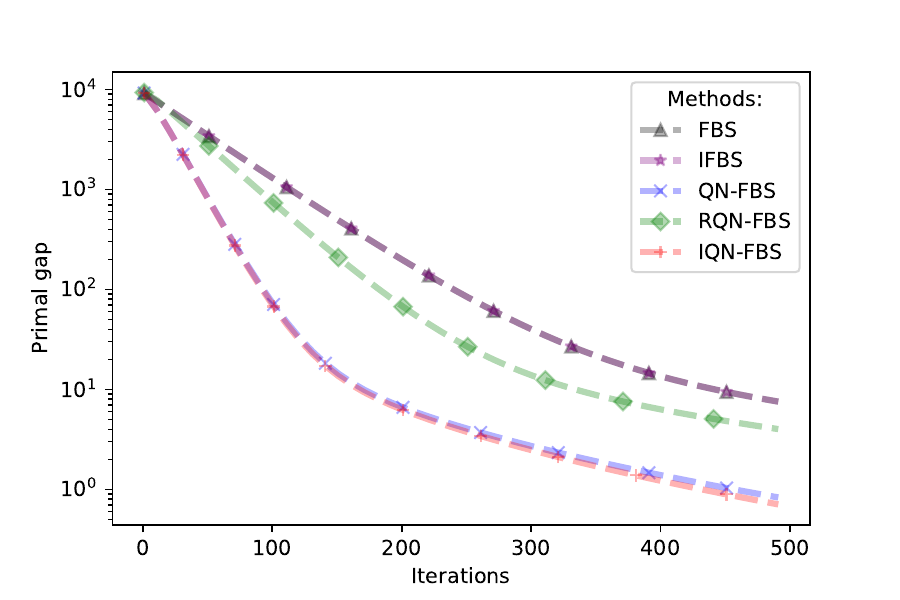}
\includegraphics[width=7cm]{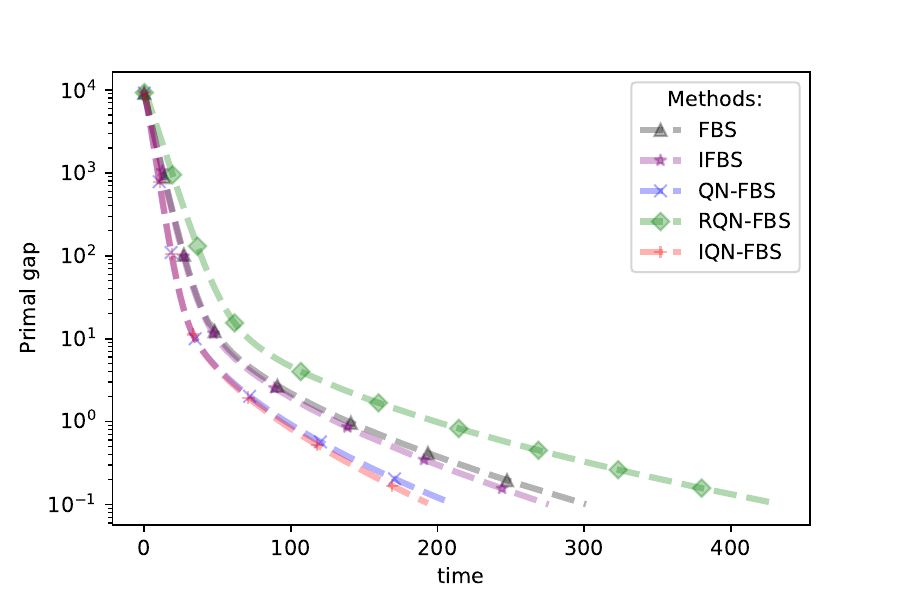}
\caption{We compare convergence of the inertial quasi-Newton PDHG (IQN-FBS) to other algorithms in the Table~\ref{tab:my-table} with $\mu = 0.5$, \textcolor{black}{$\tau =0.01\opid$}, $\sigma =0.01$ and the extrapolation parameter $\alpha_0 =1$, $\alpha_k =\min\{ \frac{10}{k^{1.1}(\max\{\norm[]{z_k-z_{k-1}},\norm[]{z_k-z_{k-1}}^2\})},1\}$. Our quasi-Newton type algorithm IQN-FBS can converge faster than the original FBS and IFBS algorithm.}
\label{fig:image-deblurring-infimalconvolution}
\end{figure}
\subsubsection{Image denoising}
We consider the same setting \eqref{inficonv:saddle}. Let $\textcolor{black}{L}=\opid$ in (\ref{inficonv:saddle}). We then obtain an image denoising problem with special norm defined by infimal convolution of total variation and weighted norm, which has strong convexity for both primal part and dual part. Besides, due to the simple formula, we obtain the dual problem explicitly which means that we can calculate primal and dual gap. The dual problem reads
\begin{equation}
    \max_{\norm[2,\infty]{y}\leq 1} - \frac{1}{2}\norm[]{D^*y-b}^2 - \frac{1}{2}\norm[]{W^{-1}y}^2\,.
\end{equation}
\begin{figure}[H]

\centering
\includegraphics[width=7cm]{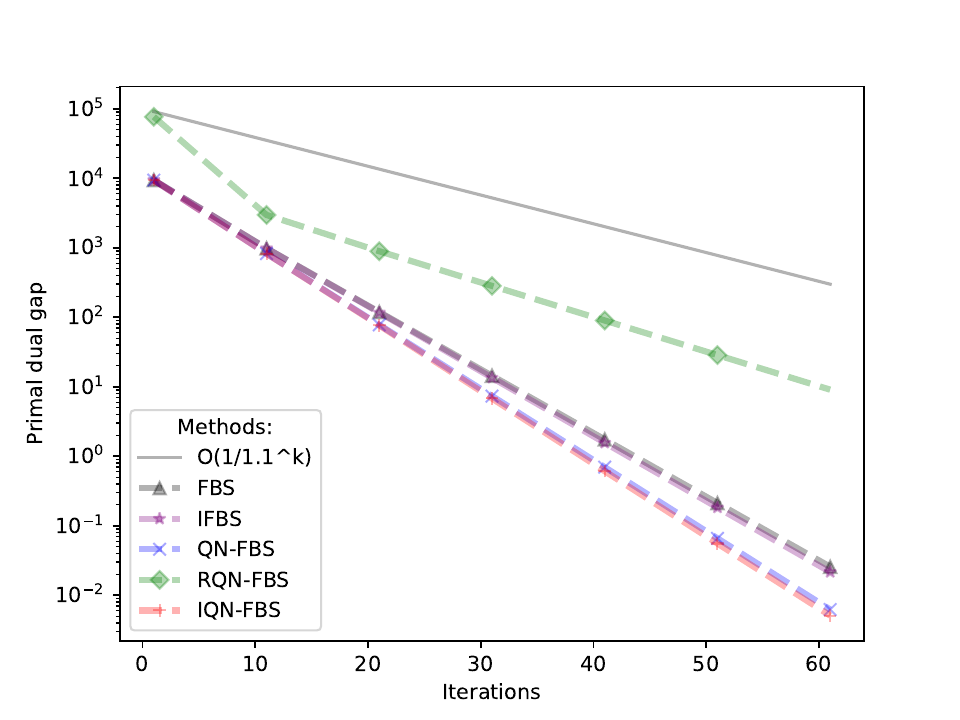}
\caption{We compare convergence of the inertial quasi-Newton PDHG (IQN-FBS) to other algorithms in the Table~\ref{tab:my-table} with $\mu = 0.1$, $\tau =0.1$, $\sigma =0.1$ and extrapolation parameter $\alpha_0 =10$, $\alpha_k = \frac{10}{\max\{k^{1.1},k^{1.1}\norm[]{z_k-z_{k-1}}^{2}\}}$. \textcolor{black}{The plot shows that all algorithms converge linearly and faster than $O(\frac{1}{1.1^k})$}.}
\label{fig:image-denoising-infimalconvolution}
\end{figure}
Figure~\ref{fig:image-denoising-infimalconvolution} shows the primal dual gap and it will decrease to zero by using any algorithm from Table~\ref{tab:my-table}. To construct $Q_k$, we use $\gamma_k=\frac{1}{\norm[]{u_k}^2}$, \textcolor{black}{$\Tau=\tau\opid=0.1\opid$ and $\Sigma=\sigma\opid=0.1\opid$}. \textcolor{black}{Assumption~\ref{Assumption:2} is satisfied. However, Assumption~\ref{assumption1} is not satisfied since we still not include the condition that $(1+\eta_k)M_k\succeq M_\kp$with $(\eta_{k})_{k\in\N}\in\ell_+^1(\N)$ in this numerical experiment}. As we can observe from Figure~\ref{fig:image-denoising-infimalconvolution}, we have linear convergence for quasi-Newton type methods as what we expected in Theorem~\ref{thm:1} and~\ref{thm:relax}. \textcolor{black}{Figure~\ref{fig:image-denoising-infimalconvolution} shows that in this experiment quasi-Newton type algorithms are in fact not more efficient compared to FBS or IFBS, which is plausible due to the well-conditioned $H=\opid$. }
\subsection{Conclusion}
In this paper, we extended the framework of \cite{becker2019quasi} for variable metrics to the setting of resolvent operators, solving efficiently the monotone inclusion problem (\ref{vi}) consisting of a set-valued operator $A$ and a co-coercive operator $B$.
We proposed two variants of quasi-Newton Forward-Backward Splitting. We develop a general efficient resolvent calculus that applies to this quasi-Newton setting. The convergence of the variant with relaxation requires  mild assumptions on the metric which are easy to satisfy, whereas the other variant implements an inertial feature and is therefore often fast. As a special case of this framework, we developed an inertial quasi-Newton primal-dual algorithm that can be flexibly applied to a large class of saddle point problems. \textcolor{black}{Throughout the paper, we employ a rank-1 perturbed variable metric denoted as $M_k=M_0+\mathrm{s}U_kU_k^*$ with $r=1$ and $\map{U_k}{\R}{\mathcal{H}}$ which is generated using the $0$-memory SR1 method. Alternatively, one can generate the variable metric using $m$-memory quasi-Newton method (refer to \cite{kanzow2022efficient, wang2023quasi}), wherein $\map{U_k}{\R^m}{\mathcal{H}}$. Consequently, we are able to derive an $m$-memory quasi-Newton primal-dual method.}

\textcolor{black}{Another potential application of our resolvent calculus in Theorem \ref{thm:resolvent} lies in non-diagonal preconditioning of the primal-dual method (PDHG). Moreover, there are many directions to improve our methods.  Further investigation is needed to develop an optimal sequence of variable metric $(M_k)_{k\in\N}$. Given that our variable metrics are currently designed solely based on the geometry of the single valued operator $B$, it is reasonable to explore a metric that adapts to both operators $A$ and $B$. Additionally, there remains an open question regarding how to eliminate the condition that the growth of $M_k$ is controlled by a summable sequence $(\eta_k)_{k\in\N}$ while ensuring fast convergence. }

\section*{Acknowledgement}
    Shida Wang, Jalal Fadili and Peter Ochs are supported by the ANR-DFG joint project TRINOM-DS under number ANR-20-CE92-0037-01, and OC150/5-1.
\begin{appendices}
\section{Appendix}
\subsection{Proof of Lemma~\ref{lem:JMT}}\label{proof:JMT}
\begin{proof}
We assume $y = M^{-1/2}\circ J_{M^{-1/2}TM^{-1/2}}\circ M^{1/2}(z)$. Then, we obtain
\begin{equation}
    \begin{split}
        y &= M^{-1/2}\circ J_{M^{-1/2}TM^{-1/2}}\circ M^{1/2}(z)\\
        M^{1/2}y &= J_{M^{-1/2}TM^{-1/2}}\circ M^{1/2}(z)\\
        M^{1/2}y &= (I+M^{-1/2}TM^{-1/2})^{-1} (M^{1/2}z )\\
        (M^{1/2}z )&\in (I+M^{-1/2}TM^{-1/2})M^{1/2}y\\
        (M^{1/2}z)&\in (M^{1/2}+M^{-1/2}T)y\\
        z&\in (\opid+M^{-1}T)y\\
        y &= (\opid+M^{-1}T)^{-1} (z )\\
        y &= J^M_T(z)\,.
    \end{split}
\end{equation}
\end{proof}
\subsection{Proof of Lemma~\ref{lemma:JMALip}}\label{PRoof:JMA}

\begin{proof} This proof is adapted from \cite{bauschke:hal-00643354}. 
Let $(u,v) = (J^{M}_A (x), J^{M}_A(y))$ for some $x,y\in \mathcal{H}$. By the definition of resolvent operator $J^{M}_A$, we obtain
\begin{equation*}
    u= J^{M}_A (x) \iff M(x-u)\in Au\,.
\end{equation*}
Similarly, we obtain $M(y-v)\in Av$. Then $\gamma_A$-strong monotonicity of $A$ yields
\begin{equation*}
\begin{split}
    \scal{M(x-u)-M(y-v)}{u-v}&\geq \gamma_A \norm[]{u-v}^2\\
    \scal{M(x-y)}{u-v}-\scal{M(u-v)}{u-v}&\geq \gamma_A \norm[]{u-v}^2\\
    \scal{M(x-y)}{u-v}&\geq \gamma_A\norm[]{u-v}^2 + \norm[M]{u-v}^2\,.\\
\end{split}
\end{equation*}
Since $\norm[]{M}$ is bounded by $C$, we obtain
\begin{equation}
    \scal{M(x-y)}{u-v} \geq \gamma_A\norm[]{u-v}^2 + \norm[M]{u-v}^2 \geq (1+\tfrac{\gamma_A}{C})\norm[M]{u-v}^2\,.
\end{equation}
Consequently, $J^{M}_A$ is $(1+\tfrac{\gamma_A}{C})$-co-coercive in the metric $M$ and Lipschitz continuous with constant $1/( 1+\tfrac{\gamma_A}{C})$ with respect to the norm $\norm[M]{\cdot}$.
\end{proof}

\subsection{Proof of Theorem~\ref{thm:1}}\label{App:proofthm1}
\begin{proof}
For convenience, we set $B_k\coloneqq M_k^{-1}B$. Fix $z\in \mathrm{zer}(A+B)$ which is equivalent to $z=J_A^{M_k}(z-B_kz)$. 
We set $\hat z_{k+1}\coloneqq J_A^{M_k}(\bar z_k - B_k\bar z_k)$, i.e., $z_{k+1}=\hat{z}_{k+1}+\epsilon_k$.\\
\textbf{Boundedness}:\\
First, we are going to show that $\norm[M_k]{z_{k}-z}$ is bounded.
The following are several useful \textcolor{black}{estimates} we will use later.
Assumption~\ref{assumption1} yields
\begin{equation}\label{ineq:Mk}
    \norm[M_{k+1}]{z_{k+1}-z}^2\leq(1+\eta_k)\norm[M_k]{z_{k+1}-z}^2\,.
\end{equation}
Since $z_{k+1}=\hat{z}_{k+1}+\epsilon_k$, it follows  that
\begin{equation}\label{ineq:error}
    \norm[M_k]{z_{k+1}-z}^2= \norm[M_k]{\hat z_{k+1} -z + \epsilon_k}^2\leq \norm[M_k]{\hat z_{k+1} -z}^2+2\norm[M_k]{\epsilon_k}\norm[M_k]{\hat z_{k+1} -z} + \norm[M_k]{\epsilon_k}^2\,.
\end{equation}
The assumption that $B$ is $\beta$-co-coercive yields that
\begin{equation}\label{bcocoercive}
    \scal{\bar z_k -z}{B_k\bar z_k - B_kz}_{M_k}=\scal{\bar z_k -z}{B\bar z_k - Bz}\geq \beta \norm[]{B\bar z_k - Bz}^2\,.
\end{equation}
The assumption that $M_k -\tfrac{1}{2\beta}\opid \in \mathcal{S}_{\kappa}(\mathcal{H})$ yields that $2\beta\opid-M_k^{-1}\in\mathcal{S}_{++}(\mathcal{H})$.\\
The co-coercivity of $B$ and monotonicity of $B$ yield the following \textcolor{black}{estimates} $\mathrm{(I)}$ and $\mathrm{(II)}$ respectively:
\begin{align*}
    \norm[M_k]{(\bar z_k - B_k\bar z_k)-(z-B_kz)}^2&= \norm[M_k]{\bar z_k-z}^2-2\scal{\bar z_k-z}{B_k\bar z_k-B_kz}_{M_k}+\norm[M_k]{B_k\bar z_k-B_kz}^2\nonumber\\
    &\overset{\mathrm{(i)}}{=} \norm[M_k]{\bar z_k-z}^2-2\scal{\bar z_k-z}{B\bar z_k-Bz}+\norm[M_k^{-1}]{B\bar z_k-Bz}^2\nonumber\\
    & \overset{\mathrm{(ii)}}{\leq}\norm[M_k]{\bar z_k -z}^2 - \norm[2\beta-M_{k}^{-1}]{B\bar z_{k} - Bz}^2  \label{est:quadratic:I}\tag{I}\\
    &\mathrm{or}\\
    &\overset{\mathrm{(iii)}}{\leq}\norm[M_k]{\bar z_k -z}^2 - \norm[\beta-M_{k}^{-1}]{B\bar z_{k} - Bz}^2-\gamma_B\norm[]{\bar z_k -z}^2\,,\label{est:quadratic:II}\tag{II} \\
\end{align*}
where $\mathrm{(i)}$ uses $\norm[M_k]{B_k\bar z_k-B_k z}^2=\norm[M_k^{-1}]{B\bar z_k-B z}^2$, $\mathrm{(ii)}$ uses (\ref{bcocoercive}) and $\mathrm{(iii)}$ uses strong monotonicity of $B$. Note that we use the shorthand $\beta-M_k^{-1}$ for $\beta\opid-M_k^{-1}$.
The fact that $J^{M_k}_A$ is firmly non-expansive since $A$ is maximally monotone with respect to $M_k$ implies that
\begin{equation}\label{non-expansive}
\begin{split}
    \norm[M_k]{\hat z_{k+1} - z}^2 &= \norm[M_k]{J^{M_k}_{A}(\bar z_k - B_k\bar z_k) -J^{M_k}_{A}(z-B_kz)}^2 \\
    &\leq \norm[M_k]{(\bar z_k - B_k\bar z_k)-(z-B_kz)}^2 \\&\quad - \norm[M_k]{(\opid-J^{M_k}_{A})(\bar z_k -  B_k\bar z_k)-(\opid-J^{M_k}_{A})(z-B_kz)}^2\\
    &\leq\norm[M_k]{\bar z_k -z}^2 - \norm[2\beta-M_{k}^{-1}]{B\bar z_{k} - Bz}^2  -\norm[M_k]{(\bar z_{k} -\hat z_{k+1})-(B_k\bar z_k - B_k z)}^2\,,\\
\end{split}
\end{equation}
where the last inequality uses \eqref{est:quadratic:I}. It follows from Assumption~\ref{assumption1} that the term $\norm[2\beta -M_k^{-1}]{B\bar z_k-Bz}^2\geq 0$. We continue to bound the first term on the right hand side of (\ref{non-expansive}).
Using \cite[Lemma 2.14]{bauschke:hal-00643354} and the definition of $\bar z_k$, we obtain the following:
\begin{equation}\label{ineq:alphaid}
    \norm[M_k]{\bar z_k -z}^2=(1+\alpha_k)\norm[M_k]{z_k - z}^2 - \alpha_k \norm[M_k]{z_{k-1} - z}^2 + (1+\alpha_k)\alpha_k \norm[M_k]{z_{k}-z_{k-1}}^2\,,
\end{equation}
and by using the triangle inequality, we also obtain another estimation:
\begin{equation}\label{ineq:alphaineq}
    \norm[M_k]{\bar z_k -z} \leq (1+\alpha_k)\norm[M_k]{z_k -z} + \alpha_k\norm[M_k]{z_{k-1}-z}\,.
\end{equation}
In order to address complete update step, we make the following estimation:
\begin{equation}\label{est:update}
    \begin{split}
        \norm[M_{k+1}]{z_{k+1}-z}^2 &\leq (1+\eta_k)\norm[M_k]{z_{k+1}-z}^2\\
        &\overset{\mathrm{(i)}}{\leq} (1+\eta_k) \big(\norm[M_k]{\hat z_{k+1} - z}^2+2\norm[M_k]{\epsilon_k}\norm[M_k]{\hat z_{k+1}-z}+\norm[M_k]{\epsilon_k}^2\big)\\
        &\overset{\mathrm{(ii)}}{\leq} (1+\eta_k) \big(\norm[M_k]{\bar z_{k} - z}^2+2\norm[M_k]{\epsilon_k}\norm[M_k]{\bar z_{k}-z}+\norm[M_k]{\epsilon_k}^2 \\&\quad- \norm[2\beta-M_{k}^{-1}]{B\bar z_{k} - Bz}^2 -\norm[M_k]{(\bar z_{k} -\hat z_{k+1})-(B_k\bar z_k - B_k z)}^2\big)\\
        &\overset{\mathrm{(iii)}}{\leq}(1+\eta_k)\big( (1+\alpha_k)\norm[M_k]{z_k - z}^2 - \alpha_k \norm[M_k]{z_{k-1} - z}^2 \\&\quad+ (1+\alpha_k)\alpha_k \norm[M_k]{z_{k}-z_{k-1}}^2 \\&\quad +2\norm[M_k]{\epsilon_k}((1+\alpha_k)\norm[M_k]{z_k -z} + \alpha_k\norm[M_k]{z_{k-1}-z}) \\&\quad + \norm[M_k]{\epsilon_k}^2- \norm[2\beta-M_{k}^{-1}]{B\bar z_{k} - Bz}^2 -\norm[M_k]{(\bar z_{k} -\hat z_{k+1})-(B_k\bar z_k - B_k z)}^2\big)\\
        &\overset{\mathrm{(iv)}}{\leq}(1+\eta_k)\big( \norm[M_k]{z_k - z}^2 \\&\quad +\alpha_k(\norm[M_k]{z_k - z}- \norm[M_k]{z_{k-1} - z})(\norm[M_k]{z_k - z}+ \norm[M_k]{z_{k-1} - z}) \\&\quad+ (1+\alpha_k)\alpha_k \norm[M_k]{z_{k}-z_{k-1}}^2 \\&\quad +2\norm[M_k]{\epsilon_k}((1+\alpha_k)\norm[M_k]{z_k -z} + \alpha_k\norm[M_k]{z_{k-1}-z}) \\&\quad + \norm[M_k]{\epsilon_k}^2 - \norm[2\beta-M_{k}^{-1}]{B\bar z_{k} - Bz}^2-\norm[M_k]{(\bar z_{k} -\hat z_{k+1})-(B_k\bar z_k - B_k z)}^2\big)\\
        &\overset{\mathrm{(v)}}{\leq}(1+\eta_k)\big( \norm[M_k]{z_k - z}^2  +\alpha_k\norm[M_k]{z_k - z_{k-1}}(\norm[M_k]{z_k - z}+ \norm[M_k]{z_{k-1} - z}) \\&\quad+ (1+\alpha_k)\alpha_k \norm[M_k]{z_{k}-z_{k-1}}^2  +2\norm[M_k]{\epsilon_k}((1+\alpha_k)\norm[M_k]{z_k -z} + \alpha_k\norm[M_k]{z_{k-1}-z}) \\&\quad + \norm[M_k]{\epsilon_k}^2 - \norm[2\beta-M_{k}^{-1}]{B\bar z_{k} - Bz}^2-\norm[M_k]{(\bar z_{k} -\hat z_{k+1})-(B_k\bar z_k - B_k z)}^2\big)\,.\\
    \end{split}
\end{equation}
where $\mathrm{(i)}$ uses (\ref{ineq:error}), $\mathrm{(ii)}$ uses (\ref{non-expansive}), $\mathrm{(iii)}$ uses (\ref{ineq:alphaid}) and (\ref{ineq:alphaineq}), $\mathrm{(iv)}$ uses factorization of the quadratic, and $\mathrm{(v)}$ uses the triangle inequality to obtain the bound $\norm[M_k]{z_k - z}-\norm[M_k]{z_{k-1}-z}\leq\norm[M_k]{z_k-z_{k-1}}$.

Now, our goal is to conclude boundedness using Lemma~\ref{lemmaPolyak}.
For simplicity, we set:
\begin{equation*}
    \begin{cases}
        e_k &\coloneqq \alpha_k\norm[M_k]{z_k-z_{k-1}}^2\\
        r_k&\coloneqq \alpha_k \norm[M_k]{z_k-z_{k-1}}\\
        \theta_k &\coloneqq \norm[M_k]{z_k-z}\\
        m_k &\coloneqq \norm[M_k]{z_{k-1}-z} \\
        p_k&\coloneqq\norm[2\beta-M_{k}^{-1}]{B\bar z_{k} - Bz}^2\\
        q_k&\coloneqq \norm[M_k]{(\bar z_{k} -\hat z_{k+1})-(B_k\bar z_k - B_k z)}^2 \,.\\
    \end{cases}
\end{equation*}
By Assumption~\ref{assumption1}, we have $m_k\leq (1+\eta_{k-1})\theta_{k-1}$. Without loss of generality, we can assume $0<\eta_k<1$ for any $k\in\mathbb{N}$. 
Replacing each term in (\ref{est:update}) with new corresponding notations, we obtain: 
\begin{equation}\label{est:update_simple}
\begin{split}
\theta_{k+1}^2&\leq (1+\eta_k)\big(\theta^2_{k}+r_k(\theta_k+ m_k) \\ &\quad+ (1+\alpha_k)e_k + 2\norm[M_k]{\epsilon_k}((1+\alpha_k)\theta_k+\alpha_km_k) +\norm[M_k]{\epsilon_k}^2 -p_k-q_k \big)\\
&\leq (1+\eta_k)\big(\theta^2_{k}+r_k(\theta_k+ (1+\eta_{k-1})\theta_{k-1}) \\ &\quad+ (1+\alpha_k)e_k + 2\norm[M_k]{\epsilon_k}((1+\alpha_k)\theta_k+\alpha_k(1+\eta_{k-1})\theta_{k-1}) +\norm[M_k]{\epsilon_k}^2 \big)\\
&\leq (1+\eta_k)\big(\theta^2_{k}+r_k(\theta_k+ 2\theta_{k-1}) + (1+
\Lambda)e_k + 2\norm[M_k]{\epsilon_k}((1+\Lambda)\theta_k+2\Lambda\theta_{k-1}) +\norm[M_k]{\epsilon_k}^2 \big)\,,\\
\end{split}    
\end{equation}
where the last inequality uses $0<\alpha_k\leq \Lambda$ and $0<\eta_k<1$.
Now, we claim that $\theta_k$ is bounded in two steps. We introduce an auxiliary bounded sequence $(C_k)_{k\in\mathbb{N}}$ (step 1) such that $\theta_k\leq C_k$ for any $k\in\mathbb{N}$ (step 2). The boundedness of $\theta_k$ follows from that of $C_k$.\\
Step1: We construct a sequence $C_k$ as the following: 
\begin{equation}\label{seq:Ck}
    \begin{cases}
        C_0 &= \max\{\theta_0,1\}\,,\\
        C_{k+1} &= (1+\eta_k)C_k + \nu_k\,,\\
    \end{cases}
\end{equation}
where $\nu_k = (1+\eta_k)((1+\Lambda)e_k+2r_k+(1+3\Lambda)\norm[M_k]{\epsilon_k})$. From our assumptions, it holds that $(M_k)_{k\in  \mathbb{N}}$ is bounded from above, $(r_k)_{k\in\mathbb{N}}\in\ell_{+}^1(\mathbb{N}), (e_k)_{k\in\mathbb{N}}\in \ell_+^1(\mathbb{N})$ and  $(\norm[]{\epsilon_k})_{k\in\mathbb{N}}\in \ell_{+}^1(\mathbb{N})$, which implies $(\nu_k)_{k\in\mathbb{N}}\in\ell_{+}^1(\mathbb{N})$. 
Using Lemma~\ref{lemmaPolyak}, we obtain the convergence of $C_k$ to some $C_{\infty}<+\infty$.\\
Step2: From the update step of (\ref{seq:Ck}), we observe that $(C_k)_{k}$ is a non-decreasing sequence and $C_k\geq 1$ for any $k\in\mathbb{N}$. We claim that for each $k$, $\theta_k \leq  C_k$. We argument by induction. Clearly, we have $\theta_0 \leq C_0$. Assume $\theta_i \leq C_i$ holds true for $i\leq k$. Then, (\ref{est:update_simple}) yields that
\begin{equation}\label{est:theta_kp1}
\begin{split}
    \theta_{k+1}^2 &\leq (1+\eta_k)\big(C_k^2+r_k(C_k+2C_{k-1})  + (1+\Lambda)e_k + 2\norm[M_k]{\epsilon_k}((1+\Lambda)C_k + 2\Lambda C_{k-1})+\norm[M_k]{\epsilon_k}^2\big)\\
    &\overset{\mathrm{(*)}}{\leq} (1+\eta_k)\big(C_k^2+4r_kC_k  + 2(1+\Lambda)e_kC_k + 2(1+3\Lambda)\norm[M_k]{\epsilon_k}C_k+\norm[M_k]{\epsilon_k}^2\big)\\
    &\leq (1+\eta_k)\big(C_k+2r_k+(1+\Lambda)e_k + (1+3\Lambda)\norm[M_k]{\epsilon_k}\big)^2\,,\\
\end{split}
\end{equation}
where $\mathrm{(*)}$ uses $C_k\geq C_{k-1}\geq 1$ and $r_k>0$. 
By the definition of $C_{k+1}$, we obtain
\begin{equation}
\begin{split}
    \theta_{k+1} &\overset{\mathrm{(i)}}{\leq} \sqrt{1+\eta_k} \big(C_k+2r_k +(1+\Lambda)e_k + (1+3\Lambda)\norm[M_k]{\epsilon_k}\big)\\
    &\overset{\mathrm{(ii)}}{\leq} (1+\eta_k)C_k+(1+\eta_k)(2r_k + (1+\Lambda)e_k + (1+3\Lambda)\norm[M_k]{\epsilon_k}) \\
    &\overset{\mathrm{(iii)}}{\leq} (1+\eta_k)C_k+\nu_k \\
    &=C_{k+1}\,,\\
\end{split}
\end{equation}
where $\mathrm{(i)}$ uses (\ref{est:theta_kp1}), $\mathrm{(ii)}$ holds true since $(1+\eta_k)>1$ and $\mathrm{(iii)}$ uses definition of $\nu_k$. This concludes the induction, and we deduce that $\theta_k$ is bounded and therefore, $ z_k$ and $\bar z_k$ are both bounded. \\
\textbf{Weak convergence}:\\
This part of the proof is adapted from the one for \cite[Theorem 4.1]{combettes2014forwardbackward}. Since $\theta_k$ is bounded, we set $\zeta\coloneqq \sup_{k\in\mathbb{N}}\theta_k$. 
The last inequality in (\ref{est:update}) implies that
\begin{equation}\label{ineq:thetasquare}
\begin{split}
    \theta_{k+1}^2&\leq(1+\eta_k)\big(\theta_k^2+r_k(\zeta+2\zeta)+(1+\Lambda)e_k+2(1+3\Lambda)\norm[M_k]{\epsilon_k}\zeta+\norm[M_k]{\epsilon_k}^2-p_k-q_k\big)\\
    &\leq(1+\eta_k)\big(\theta_k^2+3r_k\zeta+(1+\Lambda)e_k+2(1+3\Lambda)\norm[M_k]{\epsilon_k}\zeta+\norm[M_k]{\epsilon_k}^2-p_k-q_k\big)\\
    &\leq\theta_k^2+\eta_k\theta_k^2+ (1+\eta_k)\big(3r_k\zeta+(1+\Lambda)e_k+(2+6\Lambda)\norm[M_k]{\epsilon_k}\zeta+\norm[M_k]{\epsilon_k}^2\big) - p_k -q_k\\
    &\leq\theta_k^2+\underbrace{\eta_k\zeta^2+ 2(3r_k\zeta+(1+\Lambda)e_k+(2+6\Lambda)\norm[M_k]{\epsilon_k}\zeta+\norm[M_k]{\epsilon_k}^2)}_{{\delta_k}}-p_k-q_k\,.\\
\end{split}
\end{equation}
We set $\delta_k\coloneqq\eta_k\zeta^2+ 2(3r_k\zeta+(1+\Lambda)e_k+(2+6\Lambda)\norm[M_k]{\epsilon_k}\zeta+\norm[M_k]{\epsilon_k}^2)$ and observe that $(\delta_k)_{k\in\mathbb{N}}\in \ell_{+}^1(\mathbb{N})$.
Now, (\ref{ineq:thetasquare}) yields that
\begin{equation}\label{ineq:thetasquare_simple}
    \theta_{k+1}^2\leq \theta_k^2 + \delta_{k}\,.
\end{equation}
Using (\ref{ineq:thetasquare_simple}) and Lemma~\ref{lemmaPolyak}, we obtain the convergence of $\theta_k^2= \norm[M_k]{z_k-z}^2$ for any $z\in \mathrm{zer}(A+B)$.
 Rearranging (\ref{ineq:thetasquare}) to $p_k\leq \theta_k^2-\theta_{k+1}^2+\delta_{k}$, using Assumption~\ref{assumption1} and summing it for $k=0,\cdots, N$ , \textcolor{black}{there exists some $\eps\coloneqq \rho_{\min}(2\beta\opid-M_{k}^{-1})>0$ s.t.}
\begin{equation}
\begin{split}
    \eps\sum_{k=0}^N\norm[]{B\bar z_k-Bz}^2&\leq\sum_{k=0}^N\norm[2\beta-M_k^{-1}]{B\bar z_k-Bz}^2= \sum_{k=0}^N p_k\leq\theta_0^2 - \theta_N^2+\sum_{k=0}^N \delta_k \leq \zeta^2+\sum_{k=0}^N\delta_k\,.
\end{split}
\end{equation}
Since $(\delta_k)_{k\in \mathbb{N}}\in \ell^1_+(\mathbb{N})$, by taking limit as $N\to+\infty$, we obtain
\begin{equation}\label{convB}
    \sum_{k\in\mathbb{N}}\norm[]{B\bar z_k-Bz}^2\leq \frac{1}{\eps^2}(\zeta^2+\sum_{k\in\mathbb{N}}\delta_k )< +\infty\,.
\end{equation}
Similarly, we obtain from (\ref{ineq:thetasquare}) using $q_k\leq \theta_k^2-\theta_{k+1}^2 +\delta_k$ that
\begin{equation}\label{sumD}
    \sum_{k\in\mathbb{N}} \norm[M_k]{(\bar z_k-\hat z_{k+1})-(B_k\bar z_k - B_kz)}^2< +\infty\,.
\end{equation}
Set $z^*$ as an arbitrary weak sequential cluster point of $( z_k)_{k\in\mathbb{N}}$, namely, a subsequence $  z_{k_n}\weakto z^*$ as $n\to\infty$.

In order to obtain weak convergence of $ z_k$, by Proposition~\ref{fejer} with $\phi(t)=t^2$, (\ref{ineq:thetasquare_simple}) and Assumption~\ref{assumption1}, it suffices to show that $z^*\in \mathrm{zer}(A+B)$.
It follows from the selection of $\alpha_k$ that:
\begin{equation}\label{diffbarz}
\begin{split}
    \norm[]{\bar z_k - z_{k}}&\leq \alpha_k \norm[]{z_k-z_{k-1}}\to 0\,.\\
\end{split}
\end{equation}
Thus, (\ref{diffbarz}) yields $\bar z_{k_n}\weakto z^*$. From (\ref{convB}), we obtain that $B\bar z_{k_n}\to Bz$ as $n\to\infty$. Since $B$ is co-coercive, it is maximally monotone and we can use the weak strong graph closedness of $B$ in Proposition~\ref{SeqClosed} to infer that $(z^*,B z) \in \Graph B$, i.e. $Bz \in Bz^*$. However, since $B$ is single valued, we obtain $Bz^* = B z$ and hence $B\bar z_{k_n}\to Bz^*$.
Setting
    $u_k\coloneqq M_k(\bar z_{k}-\hat z_{k+1})-B\bar z_k$, by definition of the resolvent $J^{M_k}_A$, we have $u_k\in A(\hat z_{k+1})$ for all $k\in \mathbb{N}$.
From (\ref{sumD}), we obtain as $k\to +\infty$,
\begin{equation}
    \begin{split}
        \norm[]{u_k+B z^*}&=\norm[]{M_k(\bar z_k-\hat z_{k+1}-B_k\bar z_k+B_kz^*)}\\
        &\leq C\norm[]{\bar z_k-\hat z_{k+1}-B_k\bar z_k+B_kz^*}\\
        &\leq \frac{C}{\sqrt{\sigma}}\norm[M_k]{\bar z_k-\hat z_{k+1}-B_k\bar z_k+B_kz^*}\to 0\,.\\
    \end{split}
\end{equation}
Furthermore, from (\ref{convB}) and (\ref{sumD}), we have
\begin{equation}\label{z-zhat}
\begin{split}
    \norm[]{\bar z_k-\hat z_{k+1}}&\leq \norm[]{\bar z_k-\hat z_{k+1}-B_k\bar z_k +B_k z^*}+ \norm[]{B_k\bar z_k-B_kz^*}\\
    &\leq \norm[]{\bar z_k-\hat z_{k+1}-B_k\bar z_k +B_kz^*}+ \frac{1}{\sqrt{\sigma}}\norm[]{B\bar z_k-Bz^*}\to 0\,.\\
\end{split}
\end{equation}
Therefore, together with (\ref{z-zhat}),  $\bar z_{k_n}\weakto z^*$ implies $\hat z_{k_{n+1}}\weakto z^*$ as $n\to\infty$.
Now we already have $u_{k_n}\to - Bz^*$ as $n\to\infty$ and
\begin{equation}
    (\forall k \in \mathbb{N})\colon\quad(\hat z_{k_n+1},u_{k_n})\in\Graph A\,.
\end{equation}
Since $A$ is maximally monotone and using Proposition~\ref{SeqClosed}, we infer that $-Bz^* \in Az^*$, hence $z^*\in \mathrm{zer}(A+B)$. As mentioned above, the result follows from Proposition~\ref{fejer} with $\phi(t) = t^2$.\\
\textbf{Convergence rate}:\\
In the following part, we are going to show the convergence rate of Algorithm~\ref{Alg:mainAlg1}:
Assume $\epsilon_k\equiv 0$ for $k\in\mathbb{N}$ and either $\gamma_A>0$ or $\gamma_B>0$. 
Because of (\ref{ineq:Mk}), Assumption~\ref{assumption1} and Lipschitz continuity of $J_A^{M_k}$, we obtain for any $z\in \mathrm{zer}(A+B)$ that
\begin{equation}
\begin{split}
    \norm[M_{k+1}]{z_{k+1}-z}^2 &\leq (1+\eta_k)\norm[M_{k}]{z_{k+1}-z}^2\\
    &\overset{\textcolor{black}{Lemma~\ref{lemma:JMALip} }}{\leq}(1+\eta_k)\Big(\frac{1}{1+\frac{\gamma_A}{C}}\Big)^2\norm[M_k]{(\bar z_k - B_k\bar z_k)-(z-B_kz)}^2\\
    &\overset{\mathrm{(i)}}{\leq}(1+\eta_k)\Big(\frac{1}{1+\frac{\gamma_A}{C}}\Big)^2\big(   \norm[M_k]{\bar z_k -z}^2 - \norm[\beta-M_{k}^{-1}]{B\bar z_{k} - Bz}^2-\gamma_B\norm[]{\bar z_k -z}^2\big) \\
    &\overset{\mathrm{(ii)}}{\leq}(1+\eta_k)\Big(\frac{1}{1+\frac{\gamma_A}{C}}\Big)^2(1-\frac{\gamma_B}{C})\big(   \norm[M_k]{\bar z_k -z}^2 \big) \\
    &\overset{\mathrm{(iii)}}{=}(1+\eta_k)\frac{(1-\frac{\gamma_B}{C})}{(1+\frac{\gamma_A}{C})^2}\big((1+\alpha_k)\norm[M_k]{z_k-z}^2-\alpha_k\norm[M_k]{z_{k-1}-z}^2 \\&\quad+(1+\alpha_k)\alpha_k\norm[M_k]{z_k-z_{k-1}}^2 \big)\\
    &=(1+\eta_k)\frac{(1-\frac{\gamma_B}{C})}{(1+\frac{\gamma_A}{C})^2}\big(\norm[M_k]{z_k-z}^2+\alpha_k(\norm[M_k]{z_k-z}^2-\norm[M_k]{z_{k-1}-z}^2 )\\&\quad+(1+\alpha_k)\alpha_k\norm[M_k]{z_k-z_{k-1}}^2 \big)\\
    &\overset{\mathrm{(iv)}}{\leq}(1+\eta_k)\frac{(1-\frac{\gamma_B}{C})}{(1+\frac{\gamma_A}{C})^2}\big(\norm[M_k]{z_k-z}^2+\alpha_k\norm[M_k]{z_k-z_{k-1}}(\norm[M_k]{z_k-z}\\&\quad+\norm[M_k]{z_{k-1}-z})+(1+\alpha_k)\alpha_k\norm[M_k]{z_k-z_{k-1}}^2 \big)\\
    &\overset{\mathrm{(v)}}{\leq}(1+\eta_k)\frac{(1-\frac{\gamma_B}{C})}{(1+\frac{\gamma_A}{C})^2}\big(\norm[M_k]{z_k-z}^2+\alpha_k\norm[M_k]{z_k-z_{k-1}}(\norm[M_k]{z_k-z}\\&\quad+(1+\eta_{k-1})\norm[M_{k-1}]{z_{k-1}-z} )+(1+\alpha_k)\alpha_k\norm[M_k]{z_k-z_{k-1}}^2 \big)\\
    &\overset{\mathrm{(vi)}}{=}(1+\eta_k)\frac{(1-\frac{\gamma_B}{C})}{(1+\frac{\gamma_A}{C})^2}\big(\norm[M_k]{z_k-z}^2+3\alpha_k\zeta\norm[M_k]{z_k-z_{k-1}}\\&\quad+(1+\Lambda)\alpha_k\norm[M_k]{z_k-z_{k-1}}^2 \big)\\
\end{split}
\end{equation}
where $\mathrm{(i)}$ uses \eqref{est:quadratic:II}, $\mathrm{(ii)}$ uses the fact $M_k$ is bounded uniformly and the assumption that $M_k-\tfrac{1}{\beta}\opid\in\mathcal{S}_{\kappa}(\mathcal{H})$, $\mathrm{(iii)}$ uses (\ref{ineq:alphaid}), $\mathrm{(iv)}$ uses factorization of the quadratic and uses the triangle inequality to obtain the bound $\norm[M_k]{z_k-z}-\norm[M_k]{z_{k-1}-z}\leq \norm[M_k]{z_k-z_{k-1}}$, $\mathrm{(v)}$ uses Assumption~\ref{assumption1} and $\mathrm{(vi)}$ uses boundedness of $\alpha_k$ and $\norm[M_k]{z_k-z}$.
Since either $\gamma_A>0$ or $\gamma_B>0$, $\frac{1-\frac{\gamma_B}{C}}{(1+\frac{\gamma_A}{C})^2}<1$. Then there exists sufficient large $K_0>0$ such that for any $k>K_0$, $(1+\eta_k)\Big(\frac{1-\frac{\gamma_B}{C}}{(1+\frac{\gamma_A}{C})^2}\Big)<1-\xi<1$ for some $\xi\in(0,1)$. 
Thus, we infer that for any $k>K_0$:
\begin{equation}\label{alg1:convrate}
    \norm[M_{k}]{z_{k}-z}^2\leq (1-\xi)^{k-K_0}\norm[M_{K_0}]{z_{K_0}-z}^2+\sum_{i=K_0}^{k-1}(1-\xi)^{k-i}\alpha_i(3\zeta\norm[M_i]{z_i-z_{i-1}}+ (1+\Lambda)\norm[M_i]{z_i-z_{i-1}}^2)\,.
\end{equation}
Let $\Theta = 3\zeta + (1+\Lambda)$. Therefore (\ref{alg1:convrate}) can be simplified as the following:

\begin{equation}
    \norm[M_{k}]{z_{k}-z}^2\leq (1-\xi)^{k-K_0}\norm[M_{K_0}]{z_{K_0}-z}^2+\sum_{i=K_0}^{k-1}\Theta(1-\xi)^{k-i}\alpha_i\max\{\norm[M_i]{z_i-z_{i-1}}, \norm[M_i]{z_i-z_{i-1}}^2\}\,.
\end{equation}
Since $z_k$ is bounded and $M_k\in S_{\sigma}(\mathcal{H})$, it follows that 
\begin{equation}\label{convrate}
     \norm[]{z_{k}-z}^2\leq \tfrac{1}{\sigma}(1-\xi)^{k-K_0}\norm[M_{K_0}]{z_{K_0}-z}^2+O(\sum_{i=K_0}^{k-1}(1-\xi)^{k-i}\alpha_i)\,.
\end{equation}
Furthermore, if $\alpha_i\equiv 0$, for $k>K_0$ and for any $z\in \mathrm{zer}(A+B)$, we obtain linear convergence:
\begin{equation}
    \norm[]{z_{k+1}-z}^2\leq \frac{1}{\sigma} (1-\xi)^{k-K_0}\norm[M_{K_0}]{z_{K_0}-z}^2\,.
\end{equation}
If $\alpha_k\neq 0$, $\alpha_k=O(\tfrac{1}{k^2})$ and $K_0$ large enough, then $\norm[]{z_{k+1}-z}^2$ converges in the rate of $O(\tfrac{1}{k})$ for $k>K_0$ according to \cite[Lemma 2.2.4 (Chung)]{Polyak}; if $\alpha_k\neq 0$ and $\alpha_k=O(q^k)$ for $q=1-\xi$ and $k>K_0$, then $\norm[]{z_{k+1}-z}^2$ converges in the rate of $O(kq^k)$ for $k>K_0$ since (\ref{convrate}).

\qedhere
\end{proof}

\subsection{Proof of Theorem~\ref{thm:relax}}\label{proof:thm:relax}
\begin{proof}
For simplicity, we set
\begin{equation}\label{gammadelta}
    \begin{cases}
        \hat z_{k}\coloneqq J_{A}^{M_k}(z_k-M_k^{-1}Bz_k)\mathrm{,and,\ } \tilde z_k = \hat z+\epsilon_k\,,\\
        \delta_{k}\coloneqq \max\{1,\rho\}\norm[]{\epsilon_k}\,,
    \end{cases}
\end{equation}
where \textcolor{black}{$\rho=\frac{1}{2}\sqrt{\tfrac{C}{\sigma c}}{(C+\tfrac{1}{\beta})}$} and $(\delta_k)_{k\in\mathbb{N}}\in\ell_{+}^1(\mathbb{N})$. We set $z^*$ such that $-Bz^*\in Az^*$. \\
\textbf{Boundedness}:\\
We claim that by choosing proper $t_k$ for each $k\in\mathbb{N}$, we have
\begin{equation}\label{polyakseq}
    \norm[]{z_{k+1}-z^*}\leq \norm[]{z_{k}-z^*}+\delta_{k}.
\end{equation} 
Note that if (\ref{polyakseq}) is satisfied, it follows from Lemma~\ref{lemmaPolyak} that $\norm[]{z_k-z^*}$ is bounded and converges.

To stress the relation between $z_{k+1}$ and $t_k$, we define \textcolor{black}{$z(t)\coloneqq z_k-t [(M_k-B) (z_k-\tilde{z}_k)]$} and we will use $z_{k+1}$ and $z(t_k)$ interchangeably. We also set \textcolor{black}{$\gamma_k(t)\coloneqq(\delta_k+\norm{z_k-z^*})^2-\norm{z(t)-z^*}^2$}. 

In order to prove (\ref{polyakseq}), it is sufficient to show that for proper $t_k$ at each iterate, $\gamma_k(t_k)> 0$.
It results from the definition of $\gamma_k(t_k)$ that
\begin{equation}\label{gamma_kalpha_k}
\begin{split}
    \gamma_k(\textcolor{black}{t})&=(\norm{z_k-z^*}+\delta_k)^2-\norm{\textcolor{black}{z(t)}-z^*}^2\\
    &=\scal{z_k-z^*+\textcolor{black}{z(t)}-z^*}{z_k-\textcolor{black}{z(t)}} +2\delta_{k}\norm{z_k-z^*}+\delta_{k}^2\\
    &=\underbrace{2\textcolor{black}{t}\scal{z_k-z^*}{(M_k-B)(z_k-\tilde{z}_k)}}_{\mathrm{(I)}}-\underbrace{\textcolor{black}{t}^2\norm{(M_k-B)(z_k-\tilde{z}_k)}^2}_{\mathrm{(II)}} \\&\quad +2\delta_{k}\norm{z_k-z^*}+\delta_{k}^2\,.
\end{split}
\end{equation}
We need several useful properties to estimate $\mathrm{(I)}$ and $\mathrm{(II)}$.

It follows from the definition of $\hat z_k$ that $-M_k(\hat {z}_k-z_k)-Bz_k\in A\hat z_k$, and
$\gamma_A$-strong monotonicity of $A$ implies that 
\begin{equation}\label{monoA}
\begin{split}
    \gamma_A\norm[]{\hat z_k -z^*}^2&\leq\scal{\hat {z}_k-z^*}{-M_k(\hat {z}_k-z_k)-Bz_k+Bz^*}\\
    &=\scal{\hat z_k-z^*}{M_k(z_ k-\hat {z}_k)-Bz_k+B\hat  z_k-B\hat z_k+Bz^*}\\
    & = \scal{\hat z_k-z^*}{M_k(z_ k-\hat {z}_k)-Bz_k+B\hat  z_k}-\scal{\hat z_k-z^*}{B\hat z_k-Bz^*}\,.\\
\end{split}
\end{equation}
We deduce by (\ref{monoA}) and $\gamma_B$-strong monotonicity of $B$ that
\begin{equation}\label{b}
    \begin{split}
        \scal{\hat z_k-z^*}{M_k(z_ k-\hat{z}_k)-Bz_k+B\hat z_k}&\geq \scal{\hat z_k-z^*}{B\hat z_k-Bz^*} + \gamma_A\norm[]{\hat z_k - z^*}^2\\&\geq \gamma_B\norm[]{\hat z_k - z^*}^2+\gamma_A\norm[]{\hat z_k - z^*}^2\,.\\
    \end{split}
\end{equation}
$J^{M_k}_A$ is Lipschitz since $A$ is monotone.
Therefore,
\begin{equation}\label{zkz*}
\begin{split}
     \norm[M_k]{\hat z_k-z^*}^2&\leq \norm[M_k]{( z_k - B_k z_k)-(z^*-B_kz^*)}^2\\&= \norm[M_k]{z_k-z^*}^2-2\scal{ z_k-z^*}{B_k z_k-B_kz^*}_{M_k}+\norm[M_k]{B_k z_k-B_kz^*}^2\\
     &\leq \norm[M_k]{z_k-z^*}^2- \norm[2\beta-M^{-1}_k]{Bz_k-Bz}\\
     &\overset{(*)}{\leq} \norm[M_k]{z_k-z^*}^2\,,
\end{split}
\end{equation}
where $(*)$ uses Assumption~\ref{Assumption:2} that $M_k-\frac{1}{\beta}\opid\in \mathcal S_{c}(\mathcal{H})$. Using the assumption that $M_k-\frac{1}{\beta}\opid\in \mathcal S_{c}(\mathcal{H})$ again, we obtain
\begin{equation}\label{est:bk}
    \scal{z_k-\tilde z_k}{(M_k-B)(z_ k-\tilde{z}_k)}\geq \norm[M_k-\tfrac{1}{\beta}\opid]{z_k-\tilde z_k}^2 >0\,.
\end{equation}
Combining (\ref{b}), (\ref{zkz*}), the first term $\mathrm{(I)}$ in (\ref{gamma_kalpha_k}) can be estimated by the following:
\begin{equation}\label{est:b}
\begin{split}
     \mathrm{(I)}&=2\textcolor{black}{t}\scal{z_k-z^*}{(M_k-B)(z_k-\tilde z_k)}\\
     &=2\textcolor{black}{t}\scal{z_k-\tilde z_k}{(M_k-B)(z_k-\tilde z_k)}+2\textcolor{black}{t}\scal{\tilde z_k-\hat z_k}{(M_k-B)(z_k-\tilde z_k)}\\
     &\quad+2\textcolor{black}{t}\scal{\hat z_k- z^*}{(M_k-B)(z_k-\hat z_k)}+2\textcolor{black}{t}\scal{\hat z_k- z^*}{(M_k-B)(\hat z_k-\tilde z_k)}\\
     &\overset{\mathrm{(i)}}{\geq} 
     2\textcolor{black}{t}\scal{z_k-\tilde z_k}{(M_k-B)(z_k-\tilde z_k)}+2\textcolor{black}{t}\scal{\tilde z_k-\hat z_k}{(M_k-B)(z_k-\tilde z_k)}\\
     &\quad+2\textcolor{black}{t}(\gamma_A+\gamma_B)\norm[]{\hat{z}_k-z^*}^2+2\textcolor{black}{t}\scal{\hat z_k- z^*}{(M_k-B)(\hat z_k-\tilde z_k)}\\
     &\overset{\mathrm{(ii)}}{\geq} 
     2\textcolor{black}{t}\scal{z_k-\tilde z_k}{(M_k-B)(z_k-\tilde z_k)} - 2\textcolor{black}{t}\norm[]{\epsilon_k}\norm[]{(M_k-B)(z_k-\tilde z_k)}\\
     &\quad+2\textcolor{black}{t}(\gamma_A+\gamma_B)\norm[]{\hat{z}_k-z^*}^2 -2\textcolor{black}{t}\norm[M_k^{-1}]{(M_k-B)\epsilon_k}\norm[M_k]{\hat z_k-z^*}\\
     &\overset{\mathrm{(iii)}}{\geq} 
     2\textcolor{black}{t}\scal{z_k-\tilde z_k}{(M_k-B)(z_k-\tilde z_k)} - \norm[]{\epsilon_k}^2-\textcolor{black}{t}^2\norm[]{(M_k-B)(z_k-\tilde z_k)}^2\\
     &\quad+2\textcolor{black}{t}(\gamma_A+\gamma_B)\norm[]{\hat{z}_k-z^*}^2 -2\tfrac{1}{\sqrt \sigma}(C+\tfrac{1}{\beta})\textcolor{black}{t}\norm[]{\epsilon_k}\norm[M_k]{  z_k-z^*}\\
     &\overset{\mathrm{(iv)}}{\geq} 
     2\textcolor{black}{t}\scal{z_k-\tilde z_k}{(M_k-B)(z_k-\tilde z_k)} - \norm[]{\epsilon_k}^2-\textcolor{black}{t}^2\norm[]{(M_k-B)(z_k-\tilde z_k)}^2\\
     &\quad+2\textcolor{black}{t}(\gamma_A+\gamma_B)\norm[]{\hat{z}_k-z^*}^2 -2\sqrt{\tfrac{C}{\sigma}}(C+\tfrac{1}{\beta})\textcolor{black}{t}\norm[]{\epsilon_k}\norm[]{  z_k-z^*}\,,\\
\end{split}
\end{equation}
where $\mathrm{(i)}$ uses (\ref{b}), $\mathrm{(ii)}$ uses Cauchy inequality and $\mathrm{(iii)}$ uses (\ref{zkz*}), $2ab\leq a^2+b^2$ and Assumption~\ref{Assumption:2}. 
We set $b_k\coloneqq \scal{z_k-\tilde z_k}{(M_k-B)(z_ k-\tilde{z}_k)}$ and $a_k\coloneqq \norm[]{(M_k-B)(z_k-\tilde z_k)}^2$. The definition of $\delta_{k}$ yields that $\delta_k\geq \norm[]{\epsilon_k}^2$ and it follows from (\ref{est:b}) that:
\begin{equation}\label{ineq:gamma}
    \begin{split}
        \gamma_k(\textcolor{black}{t})  &\geq 2\textcolor{black}{t} b_k-2\textcolor{black}{t}^2a_k+2\delta_{k}\norm[]{z_k-z^*}+\delta_{k}^2-\norm[]{\epsilon_k}^2\\&\quad - 2\sqrt{\tfrac{C}{\sigma}}(C+\tfrac{1}{\beta})\textcolor{black}{t}\norm[]{\epsilon_k}\norm[]{  z_k-z^*}+2\textcolor{black}{t}(\gamma_A+\gamma_B) \norm[]{\hat z_k -z^*}^2\\
        &\geq \underbrace{2\textcolor{black}{t} b_k-2\textcolor{black}{t}^2a_k}_{\mathrm{(III)}}+2\textcolor{black}{t}(\gamma_A+\gamma_B) \norm[]{\hat z_k -z^*}^2 + \underbrace{2(\delta_k-\sqrt{\tfrac{C}{\sigma}}(C+\tfrac{1}{\beta})\textcolor{black}{t}\norm[]{\epsilon_k})\norm[]{  z_k-z^*}}_{\mathrm{(IV)}}\,.
    \end{split}
\end{equation}

We continue to find a proper $t_k$ such that $\gamma_k(t_k)>0$. This goal boils down to ensuring both $\mathrm{(III)}$ and $\mathrm{(IV)}$ are positive.

Let $t_k=\tfrac{b_k}{2a_k}$. We first show $t_k=\tfrac{b_k}{2a_k}$ is the proper value to make sure $\mathrm{(III)}$ is positive. 
From (\ref{est:bk}), we observe that $b_k>0$.
Since $a_k>0$ and $b_k>0$, the quadratic term $\mathrm{(III)}$ in (\ref{ineq:gamma}) will be zero for $t_k= 0$ or $t_k=\tfrac{b_k}{a_k}$ and will be strictly positive for any $t_k\in(0,b_k/a_k)$ with the maximum value obtained at $t_k=\tfrac{b_k}{2a_k}$. As a result, $\mathrm{(III)}$ is strictly positive.

Second, we will show $\mathrm{(IV)}$ is positive when $t_k=\tfrac{b_k}{2a_k}$.
We observe that $0<t_k=\tfrac{b_k}{2a_k}<\tfrac{1}{2\textcolor{black}{\sqrt{c}}}$ for Assumption~\ref{Assumption:2}.  Thus, the definition of $\delta_{k}$ and $t_k=\tfrac{b_k}{2a_k}$ imply that
\begin{equation}
\begin{split}
    \mathrm{(IV)} &= 2(\delta_k-\sqrt{\tfrac{C}{\sigma}}(C+\tfrac{1}{\beta})t_k\norm[]{\epsilon_k})\norm[]{  z_k-z^*}\\
    &\geq (2\delta_k-\sqrt{\tfrac{C}{\sigma c}}{(C+\tfrac{1}{\beta})}\norm[]{\epsilon_k})\norm[]{  z_k-z^*}\\
    &\geq0\,.
\end{split}
\end{equation}
Since $\mathrm{(III)}$ and $\mathrm{(IV)}$ both are positive when $t_k =\tfrac{b_k}{2a_k}$,
(\ref{ineq:gamma}) yields:
\begin{equation}\label{decay_gamma0}
\begin{split}
    \gamma_k(t_k)&\geq 2t_k b_k -2t_k^2a + 2t_k(\gamma_A+\gamma_B) \norm[]{\hat z_k -z^*}^2+ (2\delta_k-\sqrt{\tfrac{C}{\sigma c}}{(C+\tfrac{1}{\beta})}\norm[]{\epsilon_k})\norm[]{  z_k-z^*}\\&\geq\frac{b_k^2}{2a_k}+ 2t_k(\gamma_A+\gamma_B) \norm[]{\hat z_k -z^*}^2\\&>0\,,\\
\end{split}
\end{equation}
It results from (\ref{decay_gamma0}) and the definition of $\gamma_k(t_k)$ that for each $k\in\mathbb{N}$, $(\norm{z_k-z^*}+\delta_k)\geq \norm{z(t_k)-z^*}=\norm{z_{k+1}-z^*}$.
We conclude that if $t_k = \tfrac{b_k}{2a_k}$, then $\gamma_k(t_k)>0\ \text{for all}\ k\in \mathbb{N}$ and the sequence $\norm[]{z_k-z^*}$ is bounded and converges as $k\to+\infty$ by using Lemma~\ref{lemmaPolyak}.\\
\textbf{Weak convergence}:\\
The sequence $(z_k)_{k\in\mathbb{N}}$ generated by Algorithm~\ref{Alg:mainAlg2} is bounded and $\norm[]{z_k-z}$ converges as $k\to\infty$ and $\gamma_k(t_k)$ converges to zero as $k\to\infty$ for all $z\in \mathrm{zer}(A+B)$. Set $z^*$ as an arbitrary weak sequential cluster point of $(z_k)_{k\in\mathbb{N}}$ and there exists a subsequence $(z_{k_n})_{n\in\mathbb{N}}$ such that $z_{k_n}\weakto z^*$. 

In order to obtain weak convergence of $ z_k$, by Proposition~\ref{fejer} with $\phi(t)=t$ and fixed metric $M_k=\opid$ and (\ref{polyakseq}), it suffices to show that $z^*\in \mathrm{zer}(A+B)$.

Using Assumption~\ref{Assumption:2}, $(M_k)_{k\in\mathbb{N}}$ is bounded uniformly by $C$. Together with boundedness of operator $B$, we obtain $a_k$ is bounded by $(C+\tfrac{1}{\beta})^2\norm[]{z_k-\tilde{z}_k}^2$ for each $k\in\mathbb{N}$. Using Assumption~\ref{Assumption:2}, we obtain that $b_k\geq\norm[M_k-\tfrac{1}{\beta}\opid]{z_k-\tilde z_k}^2\geq c\norm[]{z_k-\tilde z_k}^2$. By definition of $t_k$ and \eqref{decay_gamma0}, we have $\gamma_k(t_k)\textcolor{black}{\geq \frac{b_k^2}{2a_k}}\geq \tfrac{c^2}{2(C+\tfrac{1}{\beta})^2}\norm[]{z_k-\tilde z_k}^2 $. Since $\gamma_k(t_k)\to 0$, $\norm[]{z_k-\tilde z_k}\to 0$ as $k\to\infty$. Moreover, since $\epsilon_k\to 0$, $\norm[]{z_k-\hat z_k}\to 0$. We set $u_k\coloneqq M_k(z_k-\hat z_k)+B\hat z_k - Bz_k$. Therefore, we obtain that $u_k\to 0$ as $k\to +\infty$ and $\hat z_{k_n}\weakto z^*$ as $n\to+\infty$. We observe that $u_k\in A\hat z_k+B\hat z_k$. Then, by using Proposition~\ref{SeqClosed} and the fact that $A+B$ is maximally monotone, we conclude that $0\in Az^*+ Bz^*$. \textcolor{black}{Besides}, $\norm[]{z_k-z^*}$ decreases since $z^*\in \mathrm{zer}(A+B)$.  As mentioned above, the result follows from Proposition~\ref{fejer} with $\phi(t) = t$.\\
\textbf{Linear convergence rate}:\\
If we assume $\epsilon_k\equiv 0$, then $\hat z_k=\tilde z_k$ and $\delta_{k}\equiv 0$. Therefore, from (\ref{decay_gamma0}) we can obtain an estimation for $\gamma_k(t_k)$ when $t_k=\tfrac{b_k}{2a_k}$:
\begin{equation}\label{decay_gamma}
\begin{split}
    \gamma_k(t_k)=\norm[]{z_k-z^*}^2-\norm[]{z_{k+1}-z^*}^2\geq \frac{b_k^2}{2a_k} + 2t_k(\gamma_A+\gamma_B) \norm[]{\hat z_k -z^*}^2\,.\\
\end{split}
\end{equation}
The following part is to derive linear convergence for the case that either $\gamma_A>0$ or $\gamma_B>0$. 
The definition of $b_k$ and that of $a_k$ yield the following estimation for $t_k$:
\begin{equation}\label{estimation_alpha}
    t_k = \frac{b_k}{2a_k}=\frac{\scal{z_k-\hat z_k}{(M_k-B)(z_k-\hat z_k)}}{2\norm[]{(M_k-B)(z_k -\hat z_k)}^2} \overset{(\mathrm{i})}{\geq} \frac{\norm[M_k-\tfrac{1}{\beta}]{z_k-\hat z_k}^2}{2(C + \tfrac{1}{\beta})^2\norm[]{z_k-\hat z_k}^2}\overset{(\mathrm{ii})}{>}\frac{c}{2(C+\tfrac{1}{\beta})^2},
\end{equation}
where both $(\mathrm{i})$ and $(\mathrm{ii})$ use Assumption~\ref{Assumption:2}. For convenience, we denote $\tfrac{c}{2(C+\tfrac{1}{\beta})^2}$ by $\delta$. Using Assumption~\ref{Assumption:2} again, we have the estimation for the first term at the right hand side of (\ref{decay_gamma}): 
\begin{equation}\label{estimation_bkak}
        \tfrac{b_k^2}{2a_k}\geq \delta\norm[M_k-\tfrac{1}{\beta}\opid]{z_k-\hat z_k}^2>c\delta \norm[]{z_k -\hat z_k}^2\,.
\end{equation}
Furthermore, combining (\ref{decay_gamma}) with (\ref{estimation_alpha}), (\ref{estimation_bkak}) and the definition $\gamma_k(t_k):= \norm[]{z_k-z^*}^2 - \norm[]{z_{k+1}-z^*}^2$, we obtain
\begin{equation}\label{ineq:updatedescentA}
\begin{split}
    \norm[]{z_{k+1}-z^*}^2 &\leq \norm[]{z_k-z^*}^2 - 2(\gamma_A+\gamma_B)\delta\norm[]{\hat z_k -z^*}^2- c\delta \norm[]{z_k-\hat z_k}^2 \\
    &\leq \norm[]{z_k-z^*}^2 - \tfrac{1}{2}\min\{2(\gamma_A+\gamma_B)\delta, c\delta\}(2\norm[]{\hat z_k-z^*}^2+2\norm[]{z_k-\hat z_k}^2)\\
    &\overset{(\mathrm{i})}{\leq} \norm[]{z_k-z^*}^2 - \tfrac{1}{2}\min\{2(\gamma_A+\gamma_B)\delta, c\delta\}(\norm[]{\hat z_k-z^*} + \norm[]{z_k-\hat z_k})^2\\
    &\overset{(\mathrm{ii})}{\leq}\norm[]{z_k-z^*}^2 - \tfrac{1}{2}\min\{2(\gamma_A+\gamma_B)\delta, c\delta\}\norm[]{z_k-z^*}^2\\
    &\leq (1- \tfrac{1}{2}\min\{2(\gamma_A+\gamma_B)\delta, c\delta\})\norm[]{z_k-z^*}^2\,,\\
\end{split}
\end{equation}
where $(\mathrm{i})$ uses inequality $2x^2+2y^2\geq (x+y)^2$ and $(\mathrm{ii})$ uses triangle inequality. 
Consequently, we obtain linear convergence if $(\gamma_A+\gamma_B)>0$:\begin{equation}
    \norm[]{z_k-z^*}^2\leq (1- \xi)^k \norm[]{z_0-z^*}^2\,,
\end{equation}
where $\xi=\tfrac{1}{2}\min\{2(\gamma_A+\gamma_B)\delta, c\delta\}>0$.
\end{proof}

\subsection{Proof of Theorem~\ref{thm:resolvent}}\label{PRoof:resolvent}
Now, we give the proof of Theorem~\ref{thm:resolvent}. For convenience, we define translation operator $\map{\tau_p}{\mathcal{H}}{\mathcal{H}}$ by $\tau_p(x)=x-p$ with inverse $\tau_p^{-1}=\tau_{-p}$.
\begin{proof}
Computing the resolvent operator shows the following
equivalences
\begin{equation}
\begin{split}
    &x^* = J^V_T(z) = (\opid+ V^{-1}T)^{-1}(z)\\
    \iff &Vz \in (V+T)(x^*)\\
    \iff &Mz \in (M+T)(x^*)+\mathrm{s} Q(x^*-z)\\
    [y^*=M^{1/2}x^*]\iff & Mz \in (M+T)(M^{-1/2}y^*)+\mathrm{s} Q(M^{-1/2}y^* -z)\\
    \iff &Mz \in (M+T)(M^{-1/2}y^*)+\mathrm{s} Q M^{-1/2}(y^* - M^{1/2}z)\\
    \iff & Mz \in (M^{1/2}+TM^{-1/2})(y^*)+\mathrm{s} QM^{-1/2}(y^* - M^{1/2}z)\\
    [W = M^{-1/2}QM^{-1/2}]\iff &M^{1/2}z \in (\opid+M^{-1/2}TM^{-1/2})(y^*)+\mathrm{s} W(y^* - M^{1/2}z)\,.\\
\end{split}
\end{equation}
Since \textcolor{black}{uniqueness} and existence of $x^*$ is guaranteed by the properties of $J_T^V$, Lemma~\ref{attouch} yields the existence of a unique primal-dual pair $(x^*,u^*)$ that satisfies the equivalent relations in Lemma~\ref{attouch} with $B\coloneqq +\mathrm{s} W\circ \tau_{M^{1/2}z}$ and $A\coloneqq \tau_{M^{1/2}z}\circ (\opid+ M^{-1/2}TM^{-1/2})$. The mapping $B$ is single-valued and, as $T$ is a maximally monotone operator and $M$ is positive-definite, $A^{-1}=J_{M^{-1/2}TM^{-1/2}}\circ \tau_{-M^{1/2}z}$ is single valued. Therefore, the solution of $J_T^{V}$ can be computed by finding $u^*\textcolor{black}{\in\mathrm{im}(W)}$, \textcolor{black}{namely, $u^*\in\mathrm{im}(M^{-1/2}Q)$}, such that 
\begin{equation}
    0\in B^{-1}u^* - A^{-1}(-u^*)=[(\mathrm{s} W)^{-1}u^*+M^{1/2}z] - J_{M^{-1/2}TM^{-1/2}}(M^{1/2}z - u^*)\,,
\end{equation}
and the using
\begin{equation}
    x^* = M^{-1/2}y^*\quad \mathrm{and}\quad y^*=A^{-1}(-u^*) = J_{M^{-1/2}TM^{-1/2}}(M^{1/2}z- u^*)\,.
\end{equation}
Substituting $u^*= \textcolor{black}{\mathrm{s} M^{-1/2}v^*}\textcolor{black}{\in\mathrm{im}(M^{-1/2}Q)}$ in both problems, multiplying the former one from left with $M^{-1/2}$, and \textcolor{black}{using $M^{-1/2}W^{-1}M^{-1/2}=Q^{-1}$ where $Q^{-1}$ is a set-valued inverse operator of $Q$ defined by the graph $\{(v,w)\in\mathrm{im}(Q)\times \mathcal{H}\vert Qv =w\}$} leads to \begin{equation}
    \begin{cases}
    0\in \textcolor{black}{Q^{-1}}v^*+ z - M^{-1/2}\circ J_{M^{-1/2}TM^{-1/2}}\circ M^{1/2}(z \textcolor{black}{-\mathrm{s}} M^{-1}v^*)\\
    x^* = M^{-1/2}\circ J_{M^{-1/2}TM^{-1/2}}\circ M^{1/2}(z \textcolor{black}{-\mathrm{s} } M^{-1}v^*)\,.
    \end{cases}
\end{equation}
Since a unique solution to $J^V_T$ exists, there exists $v^*\in \mathrm{im}(Q)$ that satisfies the inclusion. \textcolor{black}{We notice that $\mathrm{im}(Q^{-1})=\mathrm{im}(\mathrm{Q^+})+\mathrm{ker}Q$ where $\map{Q^+}{\mathrm{im}(Q)}{\mathrm{im}(Q)}$ is the inverse of $Q$ restricted on $\mathrm{im}(Q)$ and $\mathrm{ker}(Q)$ denotes the kernel of $Q$.} Given \textcolor{black}{the linear mapping} $\map{U}{\mathbb{R}^r}{\mathrm{im}(Q)}$, which can be realized using $r$ linearly independent $u_1,\cdots,u_r\in \mathcal{H}$ by $\alpha\to \sum_{i=1}^{r}\alpha_i u_i$, the inclusion problem is equivalent to finding the unique root $\alpha^*\in \mathbb{R}^{r}$
of $\ell(\alpha)$, \textcolor{black}{namely,
\begin{equation}
\ell(\alpha)\coloneqq U^* Q^{+} U\alpha + U^*(z - J^M_T(z-\mathrm{s} M^{-1}U\alpha))\,,
\end{equation}
} where $U^*$ denotes the adjoint of $U$ \textcolor{black}{and $J^M_T$ is an abbreviation of the mapping $M^{-1/2}\circ J_{M^{-1/2}TM^{-1/2}}\circ M^{1/2}$}. 
The following shows that $\ell(\alpha)$ is Lipschitz continuous with constant $\norm{U^*Q^{+}U} + \norm{M^{-1/2}U}^2$:
(We abbreviate $J_{M^{-1/2}TM^{-1/2}}$ by $J$ in the following)
\begin{equation}
\begin{split}
    &\scal{\ell(\alpha) -\ell(\beta)}{\alpha - \beta}\\
    =& \norm[U^*Q^{+}U]{\alpha-\beta}^2 - \scal{J(M^{1/2}z-\mathrm{s} M^{-1/2} U \alpha)-J(M^{1/2}z-\mathrm{s} M^{-1/2}U\beta)}{M^{-1/2}U(\alpha -\beta)}\\
    \leq & \norm{U^* Q^{+}U} \norm{\alpha- \beta}^2+ \norm{M^{-1/2}U}^2 \norm{\alpha-\beta}^2\,,\\
\end{split}
\end{equation}
where, in the last line, we use the 1-Lipschitz continuity (non-expansive) of $J$.

The following shows strict monotonicity of $l$. We rewrite $\ell(\alpha)$ as follows:
\begin{equation}
    \begin{split}
        \ell(\alpha)&=U^*Q^{+}U\alpha + U^*M^{-1/2}(M^{1/2}z- J_{M^{-1/2}TM^{-1/2}}(M^{1/2}z-\mathrm{s} M^{-1/2}U\alpha))\\
        &=U^*Q^{+}U\alpha+\mathrm{s} U^*M^{-1}U\alpha + U^*M^{-1/2}(\opid - J_{M^{-1/2} T M^{-1/2}})(M^{1/2}z-\mathrm{s} M^{-1/2} U \alpha)\\
        &=U^*(Q^{+}+\mathrm{s} M^{-1})U\alpha + U^*M^{-1/2}J_{M^{-1/2}T^{-1}M^{-1/2}}(M^{1/2}z-\mathrm{s} M^{-1/2}U\alpha)\,.\\
    \end{split}
\end{equation}
Using the 1-co-coercivity of $J_{M^{1/2} T^{-1} M^{1/2}}$, the function $\ell(\alpha)$ can be seen to be strictly monotone if $\alpha \mapsto U^* ( Q^{+} +\mathrm{s} M^{-1} ) U \alpha$ is strictly monotone. This fact is clear for the case $\mathrm{s}=1$.
Therefore, in the remainder, we show strictly monotonicity of $\alpha \mapsto U^* (Q^{+} - M^{-1} ) U \alpha= U^* M^{-1/2} (M^{1/2} Q^{+} M^{1/2}-\opid)M^{-1/2} U\alpha$. We observe $M-Q\in \mathcal{S}_{0}(\mathcal{H})$ implies that $\norm[]{M^{-1/2}Q M^{-1/2}}<1$ and by $1\leq \norm[]{AA^{-1}}\leq \norm[]{A}\norm[]{A^{-1}}$, we conclude that $\norm[\mathrm{im}(M^{-1/2}Q)]{M^{1/2}Q^+M^{1/2}}>1$ for the restriction of the operator norm to $\mathrm{im}(M^{-1/2}Q)$, hence, $Q^+-M^{-1}\in\mathcal{S}_{++}(\mathcal{H})$.

According to Lemma~\ref{lem:JMT}, we can replace $M^{-1/2}\circ J_{M^{-1/2}TM^{-1/2}}\circ M^{1/2}$ with $J^M_T$.
Then we obtain the formula in the statement of Theorem~\ref{thm:resolvent}.
\end{proof}

\subsubsection{Proof of Remark~\ref{remark:single}}\label{Proof:remark-single valued}\textcolor{black}{
\begin{proof}
        Given $y\in\mathrm{im}(Q)$, assume there exist $U^*x_1,U^*x_2\in U^*Q^{-1}y$ with $x_1,x_2\in Q^{-1}y$. For arbitrary $\beta\in\R^r$, we have $\scal{\beta}{U^*x_1-U^*x_2}=\scal{U\beta}{x_1-x_2}$. Since $U\beta\in \mathrm{im}Q$, then there exists some $z$ such that $U\beta = Qz$. As a result, $\scal{\beta}{U^*x_1-U^*x_2}=\scal{U\beta}{x_1-x_2}=\scal{Qz}{x_1-x_2}=\scal{z}{Qx_1-Qx_2}=0$. We notice that $\scal{\beta}{U^*x_1-U^*x_2}$ holds for arbitrary $\beta \in \R^r$. It implies $U^*x_1=U^*x_2$ and $U^*Q^{-1}y$ is single-valued.
\end{proof}
}

\subsection{Proof of Corollary~\ref{coro:resolvent}}\label{Proof:coro:resolvent}
\begin{proof}
\textcolor{black}{Let $\mathcal{H}=\R^n$. In this case, we can identify a linear mapping $\map{U}{\R^r}{\mathcal{H}}$ with a low rank matrix $U\in\R^{n\times r}$. Similarly, we can also identify $\map{U^*}{\mathcal{H}}{\R^r}$ with $U^\top\in\R^{n\times r}$.} 
Since $Q=UU^\top$ for some $U\in\mathcal{B}(\mathbb{R}^r,\textcolor{black}{\R^n})$, we have 
\begin{equation}
    \mathrm{im}(Q) =\{UU^\top v\vert v\in \textcolor{black}{\R^n}\}=\{U\alpha  \vert \alpha \in \mathbb{R}^r\}\,.
\end{equation}
Since $Q^+$ is inverse of $Q$ on $\mathrm {im}(Q)$, the following holds for arbitrary $v\in \textcolor{black}{\R^n}$:
\begin{equation}
    QQ^{+}Qv = Qv\iff UU^\top Q^{+}UU^\top v=UU^\top v \iff UU^\top Q^{+}U\alpha =U\alpha\,.
\end{equation}
Since the column vectors $\{u_i\}_{i=1,\cdots,r}$ of $U$ are independent with each other, $UU^\top Q^{+}U\alpha =U\alpha$ yields that $ U^\top Q^{+}U\alpha =\alpha$.
Therefore, the root-finding problem in Theorem~\ref{thm:resolvent} simplifies to (\ref{eq:la}). 
\end{proof}

\subsection{Proof of Proposition~\ref{semi-smooth-Newton}}\label{Proof:semi-smooth-Newton}
\begin{proof}
Our proof relies on the convergence result \cite[Theorem 7.5.5]{facchinei2003finite}. By the same argument as the one in Appendix B.5 paper \cite{becker2019quasi}, we obtain $\partial^C l(\alpha^*)$ is non-singular. If $\ell(\alpha)$ is tame, then by \cite[Theorem 1]{bolte2009tame}, $\ell(\alpha)$ is semi-smooth. In order to apply this result it remains to show that $\ell(\alpha)$ is tame.
The property of definable functions is preserved by operations including the sum, composition by a linear operator, derivation and canonical projection (\cite{van1996geometric}, \cite{coste2000introduction}). Since $T$ is a tame mapping, $I+M^{-1/2}TM^{-1/2}$ is tame as well as its graph. Here $I$ is identity. Then the resolvent  $J_{M^{-1/2}TM^{-1/2}}=(I+M^{-1/2}TM^{-1/2})^{-1}$ which is defined by the inverse of the same graph is tame (\cite{ioffe2009invitation}) and single-valued. By the stability of the sum and composition by linear operator, we obtain that $\ell(\alpha)$ is tame.  
\end{proof}
\subsection{Proof of Proposition~\ref{prop:bound}}\label{Proof:bound}
\begin{proof}
Let $p=J^{V}_T(z)$. Since resolvent operator is non-expansive with respect to \textcolor{black}{$V$}, \textcolor{black}{$\norm[V]{p}=\norm[V]{J^{V}_T(z)}\leq \norm[V]{z}+\norm[V]{J_{T}^V(0)}$}. By duality, optimal $\alpha^*$ will satisfy $ \alpha^* =\textcolor{black}{u^\top} (p-z)$. Then, 
\begin{equation}
\begin{split}
    |\alpha^*|&=|\textcolor{black}{u^*} (p-z)|\\
    &\leq \textcolor{black}{\norm[V^{-1}]{u}(2\norm[V]{z} + \norm[V]{J^V_T(0)})}\,.\\
\end{split}
\end{equation}
If $V\in\mathcal{S}_c(\R^n)$ is bounded by a constant $C$, then
\begin{equation}
    |\alpha^*|\leq \textcolor{black}{\norm[V^{-1}]{u}(2\norm[V]{z} + \norm[V]{J^V_T(0)})}\leq \textcolor{black}{\frac{C}{c}\norm[]{u}(2\norm[]{z} + \norm[]{J^V_T(0)})}\,.\quad\qedhere \qedsymbol
\end{equation}
\end{proof}
\subsection{Proof of Proposition~\ref{prop:PDHGnew}}
\begin{proof}\label{Proof:PDHGnew}
$\mathcal{L}(\alpha)$ is as defined in Proposition~\ref{prop:PDHG}. Substituting \textcolor{black}{$\alpha = \xi + U^* V^{-1}Bz$ in $\mathcal{L}(\alpha)$,} we obtain $\mathcal{J}(\xi)=\mathcal{L}(\alpha)$. \textcolor{black}{Then, there exists $\xi^*$ such that $\alpha^* = \xi^* + U^* V^{-1}Bz$.} In (\ref{resolventJM}), we do the same substitution.
\begin{equation}
\begin{split}
    \hat z &= J_T^M(\breve{z} -\mathrm{s}  M^{-1}U\alpha^*) \\
    & = J_A^M(\breve{z} -\mathrm{s}  M^{-1}U\xi^*-\mathrm{s} M^{-1}U \textcolor{black}{U^*} V^{-1}Bz)\\
    & = J_A^M(z - M^{-1}MV^{-1}Bz-\mathrm{s} M^{-1}U\textcolor{black}{U^*} V^{-1}Bz -\mathrm{s}  M^{-1}U\xi^*)\\
    & = J_A^M(z - M^{-1}\underbrace{(M+\mathrm{s} U\textcolor{black}{U^*} )}_{=V}V^{-1}Bz -\mathrm{s}  M^{-1}U\xi^*)\\
    & = J_A^M(z - M^{-1}Bz -\mathrm{s}  M^{-1}U\xi^*)\,.\\
\end{split}
\end{equation}
Due to Proposition~\ref{prop:PDHG},  $\mathcal{L}(\alpha)$ is Lipschitz with constant $1+\norm{M^{-1/2}U}^2$ and strongly monotone. \textcolor{black}{Since $\mathcal{J}(\xi)$ is obtained by translation, it enjoys the same properties.}
\end{proof}
\subsection{Proof of Lemma~\ref{lemma:PDHGcond1}}\label{Proof:PDHGcond1}
\begin{proof}
    \begin{enumerate}[(i)]
        \item $M_0$ is symmetric positive-definite and $\gamma_k u_ku_k^*$ is symmetric positive semi-definite since $\inf_{k \in \N}\gamma_k>0$. Thus, $M_k \succeq M_0$ showing that $M_k$ is  symmetric positive-definite. Moreover, $M_k-\frac{1}{\beta}\opid \succeq M_0-\frac{1}{\beta}\opid \succeq c \opid$, which shows part (i) of Assumption~\ref{Assumption:2}. We now show that $M_k$ obeys (ii) of the assumption. Since $B$ is $\beta$-co-coercive, we have in view of \cite[Remark~4.34 and Proposition~4.35]{bauschke:hal-00643354} that
        \begin{equation}\label{eq:cocoercivebnd}
        \beta\norm{y_k}^2 \leq \scal{y_k}{s_k} \leq \norm{s_k}^2/\beta .
        \end{equation}
        Assume that $s_k \neq 0$ (otherwise, there is nothing to prove). Thus
        \[
        \scal{M_0 s_k - y_k}{s_k} = \norm[M_0]{s_k}^2 - \scal{y_k}{s_k} \geq \norm[M_0]{s_k}^2 - \norm{s_k}^2/\beta = \norm[M_0-\beta\opid]{s_k}^2 \geq c \norm{s_k}^2 .
        \]
        Combining this with \eqref{eq:cocoercivebnd}, we get
        \begin{equation}\label{eq:ukbnd}
        \norm{u_k} = \frac{\norm{y_k-M_0s_k}}{\sqrt{\scal{M_0 s_k - y_k}{s_k}}} \leq \frac{\norm{y_k} + \norm{M_0}\norm{s_k}}{\sqrt{c}\norm{s_k}} \leq \frac{1/\beta + \norm{M_0}}{\sqrt{c}} .
        \end{equation}
        This entails that 
        \[
        \sup_{k \in \N}\norm{M_k} \leq \norm{M_0} + \sup_{k \in \N}\gamma_k\norm{u_k}^2 \leq \norm{M_0} + \frac{(1/\beta + \norm{M_0})^2}{c} \sup_{k \in \N}\gamma_k < +\infty .
        \]
        \item Let us focus on part (i) of Assumption~\ref{Assumption:2}. We have, using \eqref{eq:ukbnd},
        \[
        M_k-\frac{1}{\beta}\opid \succeq M_0-\frac{1}{\beta}\opid -\gamma_k\norm{u_ku_k^*}\opid=(c - \gamma_k\norm{u_k}^2)I \succeq \left(c-\frac{(1/\beta + \norm{M_0})^2}{c}\gamma_k\right)
        \]
        and the last term is positive under the prescribed choice of $\gamma_k$. To verify part (ii), it is sufficient to observe that $\norm{M_k} \leq \norm{M_0}$.
    \end{enumerate}
\end{proof}
\subsection{Proof of Lemma~\ref{Lemma:PDHGcond2}}\label{Proof:PDHGcond2}
\begin{proof}
\begin{enumerate}[(i)]
\item As in the proof of Lemma~\ref{lemma:PDHGcond1},  $M_k \succeq M_0 \succeq (1/\beta+c)\opid > 0$. Moreover
\begin{align*}
(1+\eta_k)M_k - M_{k+1}
&= \eta_k M_0 + (1+\eta_k)\gamma_ku_ku_k^* - \gamma_{k+1}u_{k+1}u_{k+1}^* \\
&\succeq \eta_k (1/\beta+c)\opid - \gamma_{k+1}\norm{u_{k+1}}^2 \\
&= \eta_k (1/\beta+c)\opid - \eta_{k+1}(1/\beta+c)\opid \\
&= (1/\beta+c)(\eta_k - \eta_{k+1})\opid \succeq 0 ,
\end{align*}
since $\eta_k$ is non-increasing. The uniform boundedness of $M_k$ is straightforward as $\norm{M_k} \leq \norm{M_0} + \eta_k (1/\beta+c)$.
\item We have in this case
\[
M_k-\frac{1}{\beta}\opid \succeq M_0-\frac{1}{\beta}\opid -\gamma_k\norm{u_ku_k^*}\opid \succeq \left(c - \eta_k\kappa(1/\beta+c)\right)\opid \succeq \left(c - (1-\kappa)(1/\beta+c)\right)\opid 
\]
Under our condition on $\kappa$, we have $c - (1-\kappa)(1/\beta+c > 0$. In addition, 
\begin{align*}
(1+\eta_k)M_k - M_{k+1}
&= \eta_k M_0 - (1+\eta_k)\gamma_ku_ku_k^* + \gamma_{k+1}u_{k+1}u_{k+1}^* \\
&\succeq \eta_k (1/\beta+c)\opid - (1+\eta_k)\eta_k\kappa(1/\beta+c)\opid \\
&= \eta_k (1/\beta+c)(1-(1+\eta_k)\kappa)\opid \succeq 0
\end{align*}
since $\eta_k \leq (1-\kappa)/\kappa$. $M_k$ is also uniformly bounded with the same argument as above. This completes the proof.
\end{enumerate}
\end{proof}

\end{appendices}
\section*{Data of availability}
In order to record the exact algorithmic details for our experiments, all code for experiments from this paper is available at \url{https://github.com/wsdxiaohao/quasi_Newton_FBS.git}.
\section*{Conflict of interest statements}
This work was supported by Deutsche Forschungsgemeinschaft (Grant numbers OC150/5-1) and Agence Nationale de la Recherche (Grant numbers ANR-20-CE92-0037-01).
\bibliographystyle{siam}
\bibliography{main}
\end{document}